\documentclass[a4paper,11pt]{article}

\usepackage{anyfontsize}

\usepackage{booktabs}
\usepackage{colortbl}
\usepackage[pdftex,dvipsnames,table]{xcolor}
\usepackage{xfrac}
\usepackage{xargs}
\usepackage[utf8]{inputenc}
\usepackage[mathscr]{euscript}
\usepackage{xspace}
\usepackage{graphicx}
\usepackage{color}
\usepackage{enumerate}
\usepackage[english]{babel}
\usepackage{verbatim}
\usepackage{float}
\usepackage{soul}
\usepackage{rotating}
\usepackage{caption}
\usepackage[draft]{fixme}
\usepackage{longtable}
\usepackage{mathtools}
\usepackage[shortlabels]{enumitem}
\usepackage{tabularx}
\usepackage{bbm}
\usepackage[inkscapelatex=false]{svg}
\usepackage{subfig}

\definecolor{lightgray}{gray}{0.9}
\restylefloat{table}

\makeatletter
\newcommand{\multiline}[1]{%
  \begin{tabularx}{\dimexpr\linewidth-\ALG@thistlm}[t]{@{}X@{}}
    #1
  \end{tabularx}
}
\makeatother

\usepackage{algorithm}%
\usepackage[noend]{algpseudocode}%

\newcommand{\R}{\mathbb{R}}
\newcommand{\Z}{\mathbb{Z}}

\newcommand{\Q}{\mathbb{Q}}

\newcommand{\mb}[1]{\mathbb{#1}}

\newcommand{\tsc}[1]{\textsc{#1}}

\newcommand{\oa}[1]{\vec{#1}}
\newcommand{\allones}{\mathbbm{1}}

\newcommand{\recc}{\textsc{recc}}
\newcommand{\conv}{\textsc{conv}}
\newcommand{\cone}{\textsc{cone}}

\newcommand{\vertex}{\textsc{vertex}}

\newcommand{\edge}{\textsc{edge}}

\newcommand{\epi}{\textsc{epi}}
\newcommand{\dom}{\textsc{dom}}

\makeatletter
\DeclareRobustCommand{\cev}[1]{%
  {\mathpalette\do@cev{#1}}%
}
\newcommand{\do@cev}[2]{%
  \vbox{\offinterlineskip
    \sbox\z@{$\m@th#1 x$}%
    \ialign{##\cr
      \hidewidth\reflectbox{$\m@th#1\vec{}\mkern4mu$}\hidewidth\cr
      \noalign{\kern-\ht\z@}
      $\m@th#1#2$\cr
    }%
  }%
}
\makeatother

 \let\mathscr\relax%
\usepackage[scr]{rsfso}

\usepackage{stackengine,scalerel}
\stackMath

\newcommand{\T}{\mathsf{\scriptscriptstyle T}} %
\newcommand*{\defeq}{\mathrel{\vcenter{\baselineskip0.5ex \lineskiplimit0pt
                     \hbox{\footnotesize.}\hbox{\footnotesize.}}}%
                     =}

\newcommand{\suchthat}{: }

\newcommand{\fYnoref}[1][Y]{f_{#1}}

\newcommand{\fYx}[1][Y]{\fYnoref}
\newcommand{\fY}[1][Y]{\hyperref[definition:value_function]{\fYnoref[#1]}}
\newcommand{\domfY}[1][Y]{\dom(\fY[#1])}
\newcommand{\epifY}[1][Y]{\epi(\fY[#1])}
\newcommand{\revpolar}{\hyperref[epigraph_rev_polar]{E^{\#}}}
\newcommand{\feasreg}{\hyperref[eq:feasible_region]{\mathcal{F}}}
\newcommand{\nodearcmatrix}[1][N]{\mathbf{#1}}
\newcommand{\cardinality}[1]{\lvert #1 \rvert}
\newcommand{\SRC}{\textsc{\texttt{SRC}}}
\newcommand{\FWD}{\textsc{\texttt{FWD}}}
\newcommand{\SNK}{\textsc{\texttt{SNK}}}
\newcommand{\textsuchthat}{}
\newcommand{\solvereversepolar}{\mbox{\hyperref[algorithm:solve_polar]{\tsc{SolveReversePolar}}}\xspace}
\newcommand{\separatecorner}{\mbox{\hyperref[algorithm:separate_corner]{\tsc{SeparateCornerBendersCuts}}}\xspace}
\newcommand{\cglp}{\mbox{\hyperref[alg_solve_polar:cglp]{\tsc{(CGLP)}}}\xspace}

\makeatletter
\newcommand{\pushright}[1]{\ifmeasuring@#1\else\omit\hfill$\displaystyle#1$\fi\ignorespaces}
\newcommand{\pushleft}[1]{\ifmeasuring@#1\else\omit$\displaystyle#1$\hfill\fi\ignorespaces}
\makeatother

\usepackage{rotating}
\usepackage{fancyvrb}

\usepackage{hyperref}
\usepackage{cleveref}

\makeatletter
\newcounter{HALG@line}
\renewcommand{\theHALG@line}{\thealgorithm.\arabic{ALG@line}}
\makeatother

\newcommand{\mytitle}[1]{#1}
\usepackage[margin=1.2in]{geometry}
\usepackage{amsmath,amsthm, amssymb, amsfonts, amsbsy}
\usepackage{thmtools} %
\usepackage{thm-restate} %
\usepackage{authblk}
\providecommand{\keywords}[1]{\textit{Keywords:} #1}
\usepackage{natbib}
\setcitestyle{authoryear,round}

\DeclareMathOperator*{\argmin}{arg\,min}

\newcommand{\Halmos}{\qed}
\newenvironment{myproof}[1]{%
\par\vspace{0.5cm}\noindent{\it Proof #1. }%
}{\hfill \Halmos}

\theoremstyle{plain}
\newtheorem{theorem}{Theorem}
\newtheorem{proposition}{Proposition}

\newtheorem{lemma}{Lemma}
\newtheorem{claim}{Claim}
\newtheorem{fact}{Fact}

\theoremstyle{definition}

\newtheorem{definition}{Definition}

\newtheorem{remark}{Remark}

\newenvironment{assumption*}[1][Assumption]{\par\vspace{0.5cm}\noindent\textbf{#1}}{}
\newenvironment{APPENDICES}{\appendix}{}

\begin{document}
\title{Approximating value functions via corner Benders' cuts}

\author[1]{Matheus J. Ota\thanks{(mjota@uwaterloo.ca)}}
\author[1]{Ricardo Fukasawa\thanks{(rfukasawa@uwaterloo.ca)}}
\author[2]{Aleksandr M. Kazachkov\thanks{(akazachkov@ufl.edu)}}
\affil[1]{Department of Combinatorics and Optimization, University of Waterloo, Waterloo, Ontario, Canada}
\affil[2]{Department of Industrial and Systems Engineering, University of Florida, Gainesville, Florida, United States of America}

\maketitle

\begin{abstract}
We introduce a novel technique to generate Benders' cuts from a conic relaxation (``corner'') derived from a basis of a higher-dimensional polyhedron that we aim to outer approximate in a lower-dimensional space. To generate facet-defining inequalities for the epigraph associated to this corner, we develop a computationally-efficient algorithm based on a compact reverse polar formulation and a row generation scheme that handles the redundant inequalities. Via a known connection between arc-flow and path-flow formulations, we show that our method can recover the linear programming bound of a Dantzig-Wolfe formulation using multiple cuts in the projected space. In computational experiments, our generic technique enhances the performance of a problem-specific state-of-the-art algorithm for the vehicle routing problem with stochastic demands, a well-studied variant of the classic capacitated vehicle routing problem that accounts for customer demand uncertainty.
\end{abstract}


\section{\mytitle{Introduction}}
\label{section:intro}

Optimizing a large-scale \emph{linear program} (LP) often relies on \emph{decomposition} methods that 
distinguish ``easier-to-solve'' structure in the problem from ``complicating'' subproblems, repeatedly deriving information from the subproblems as needed until eventually converging to optimality~\citep{vanderbeck2010reformulation}.
We focus on improving \emph{Benders' decomposition}~\citep{Benders1962}, a method that partitions the LP variables
into \emph{master variables} and \emph{subproblem variables}, which we may also refer to as \emph{first-stage} and \emph{second-stage} variables. In this approach, an initial convex relaxation is formulated in the space of the first-stage variables and is iteratively refined with \emph{Benders' cutting planes (cuts)}, which are inequalities derived from solving the second-stage subproblems.

Benders' cuts relay constraints and costs from the subproblems by projecting them into the space of the master decision variables. This projection step requires the solution of an LP in the space of the subproblem variables and the subsequent cut is derived based on the LP optimal primal/dual solutions (i.e., the Benders' subproblem). While these cuts suffice for eventual convergence, the procedure may require a high overall computational time~\citep{rahmaniani2017benders}.
Our starting point is the observation that using only an optimal solution of a Benders' subproblem discards other potentially relevant information, such as the corresponding optimal basis.
This inspires our main contributions:
\begin{itemize}[itemsep=0pt,leftmargin=*]
    \item We propose \emph{corner Benders' cuts} to \textbf{capture more subproblem structure from basis information} and accelerate the convergence of typical Benders' decomposition methods.
    \item We argue that, even though the corner Benders' cuts are obtained from a relaxation of the original LP, any regular Benders' cut is in fact a corner Benders' cut. 
    \item We show how one can \textbf{efficiently obtain and generate corner Benders' cuts} based on objective function cuts (i.e., inequalities that impose a lower bound on the objective function value), attaining the same LP bound implied by those cuts,
    but without the drawback of using objective-parallel inequalities.
    \item We computationally demonstrate that the ideas derived lead to \textbf{state-of-the-art branch-and-cut algorithm performance} for the \emph{vehicle routing problem with stochastic demands}.
\end{itemize}

\noindent
As far as we are aware, this is the first work that leverages basis information to achieve better computational performance of Benders' cuts without exploiting integrality constraints.

\subsection{\mytitle{Problem setup}}

Concretely, we consider the following optimization problem over variables $x \in X \subseteq \R^n$ and $y \in Y \subseteq \R^m$, where $X$ and $Y$ are nonempty polyhedra defined by rational data, with~$Y$ assumed to be pointed. The variables are coupled by~$p$ equations, $Tx + Qy = h$, referred to as \emph{linking constraints}:
\begin{equation}
    \tag{LP}
    \label{problem:basic}
    \begin{aligned}
        z^* \defeq \min_{x,y} \quad &c^\T x + d^\T y \\
        & T x + Q y = h, \\
        & (x, y) \in X \times Y.
    \end{aligned}
\end{equation}
While we do not include any integer restrictions on the variables,
problem~\eqref{problem:basic} arises in the discrete setting as well, when solving a relaxation of a mixed-integer LP instance,
which is the context of our computational experiments.

We define the first-stage problem in the $x$ variables and we treat~$Y$ as a polyhedron associated with the second-stage subproblem.%
\footnote{The set $Y$ could encompass a block structure with multiple smaller subproblems. For an extension to this setting, see Appendix~\ref{appendix:block_diagonal}.}
Benders' decomposition projects out the $y$ variables via cuts that approximate the second-stage \textit{value function} given fixed first-stage decisions $\bar{x}$:
\begin{equation*}
    f_Y(\bar{x}) \defeq
        \min_{y} \{ d^\T y \suchthat Qy = h - T \bar{x},\, y \in Y \}.
\end{equation*}
This leads to the following reformulation of \eqref{problem:basic} (with same optimal objective value $z^*$) in the space of the first-stage variables and an additional \emph{epigraph} variable $\theta$:
\begin{equation}
    \min_{x,\theta} \{ c^\T x + \theta \suchthat \theta \ge \fYx(x),\, x \in X \}.
    \label{problem_master}
\end{equation}

\subsection{\mytitle{Intuition for corner Benders' cuts}}

The purpose of Benders' cuts is to provide an outer approximation of Problem~\eqref{problem_master} with linear inequalities. In prior work, a central focus for improving Benders' decomposition has been choosing an appropriate normalization for a cut-generating linear program associated to the dual formulation of $\fYx(\bar{x})$ for a given candidate solution $(\bar{x},\bar{\theta})$ violating the epigraph constraint in Problem~\eqref{problem_master}~\citep{fischetti2010note, bonami2020implementing, conforti2019facet, brandenberg2021refined, hosseini2024deepest}. We instead concentrate on a new \emph{approach} to Benders' cuts. We note that there have been many other works that deal with Benders' cuts that use the integrality of some of the variables, for example~\citep{kuccukyavuz2017introduction, gade2014decomposition, zou2019stochastic, rahmaniani2020dualdecomp, chen2022generating}. We refrain from doing a thorough literature review of those types of works since our improvement is based solely on \eqref{problem:basic}, which lacks integer restrictions on the variables.

The core motivation for our approach is computational:
a single Benders' cut is generated by solving an LP in the $Y$ space and provides an outer approximation of Problem~\eqref{problem_master},
but the same LP also offers basis information from which we will derive a set of Benders' cuts and obtain a tighter approximation of the value function.
For efficiency in computing these cuts, we replace $Y$ by a much simpler relaxation $C$.
This idea is illustrated in Figure~\ref{figure:epigraph},
where~$f_C$ denotes the function obtained by replacing~$Y$ with~$C$ in the definition of~$\fYx$.
Specifically, given a basic feasible solution to $\fYx(\bar{x})$, we use for $C$ the \emph{basis cone}, which (perhaps more evocatively) is also known as the \emph{relaxed (Gomory's) corner polyhedron}---an object often used for deriving cuts for integer programs~\citep{ConCorZam11_corner-polyhedron-and-intersection-cuts}---defined by the subset of constraints that are tight at an optimal solution and associated to the nonbasic variables~\citep{Gomory65_on-the-relation-between-integer-noninteger-solutions,gomory1969some}. We show that every Benders' cut is valid for the epigraph of~$f_C$, for some appropriately chosen basis cone~$C$. Moreover, in our computational experiments, when $C$ is selected according to an objective function cut~$c^\T x + \theta \ge z^*$ (as discussed in Section~\ref{subsection:corner_bounds}), generating valid inequalities for the epigraph of~$f_C$ exhibits better empirical performance than the standard Benders' decomposition approach.

We are aware of two previous works that use basis information for two-stage stochastic programs~\citep{gade2014decomposition, romeijnders2016convex}, both of which focus on the case with integrality constraints on the second-stage variables. In~\cite{gade2014decomposition} the authors use basis information to generate parametric Gomory cuts for the second-stage problem. In contrast,~\cite{romeijnders2016convex} use a convex approximation of the second-stage value function and leverage periodicity properties of value functions associated with corner polyhedra (maintaining integrality on the variables) to bound the approximation error. Despite also using basis information, both approaches are fundamentally different from our proposed method.

\begin{figure}
	\centering
	\includegraphics[scale=1.1]{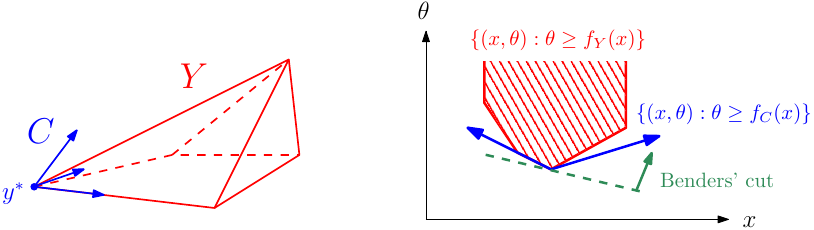}
	\caption{
        The basis associated to a vertex $y^*$ of $Y$ yields the translated cone (the relaxed corner polyhedron) $C \supseteq Y$, shown in the left panel.
        The facets of $C$ can then be projected into the~$(x, \theta)$-domain, depicted in the right panel.
    }
    \label{figure:epigraph}
\end{figure}

\subsection{\mytitle{Network flow structures}}

While our overall framework applies generically,
when we further assume that $Y$ is a network flow polytope,
we obtain additional theoretical insights and algorithmic enhancements. 
This is achieved by using the known combinatorial structure of bases of network flow problems~\citep{ahuja1988network}, which allows us to accelerate the Benders' decomposition process and derivation of corner Benders' cuts (see Section~\ref{section:vrpsd} and Appendix~\ref{appendix:corner_network}).

A network flow structure in the second-stage decisions is encountered frequently and captures important classes of problems.
For example, it appears in numerous transportation and logistics applications that have been shown to benefit from Benders' decomposition~\citep{costa2005survey, rahmaniani2017benders}.

A network flow polytope can also be used to efficiently encode the convex hull of a finite set of vectors~\citep{de2022arc}.
This principle underlies how \emph{decision diagrams} can offer strong relaxations of combinatorial optimization problems~\citep{becker2005, behle_thesis, BerCirHoeHoo16_decision-diagrams, tjandraatmadja2019target, castro2022combinatorial, lozano2022constrained}, which can be used to model second-stage problems for two-stage stochastic programs~\citep{lozano2022binary}.

The convex hull subproblem is also a fundamental aspect of \emph{Dantzig-Wolfe (DW) decomposition}, for which the partitioning scheme is based on a complicating subsystem of constraints, rather than of the variables as in Benders' decomposition.
Satisfying the complicating constraints can be rewritten as requiring that solutions can be expressed as a convex combination of extreme points (in the case that the subproblem is a polytope),
which gives rise to a network flow structure,
as we discuss in Section~\ref{subsection:path_flow}. This is reflected in~\eqref{problem:basic} by treating the $x$ space as ``original'' problem variables and the $y$ space as flow variables coming from a DW reformulation.

In this context, \cite{chen2024recovering} recently presented a strategy to replace an objective function cut that recovers the \emph{DW bound} (i.e., the LP bound of a DW reformulation) by a family of cuts in the original variables that are not parallel to the objective function. In that sense, the approach of \cite{chen2024recovering} shares some of the same ideas as our work. In fact, our computational experiments are framed in the context of  adding corner Benders' cuts based on DW decomposition bounds.
Our approach of taking advantage of basis information,  provides computational advantages compared to the method in \cite{chen2024recovering}, which does not exploit that.

We emphasize that, as highlighted in~\cite{uchoa2024optimizing}, a motivation for recovering the DW bound through cuts is that many commercial optimization solvers have native support for the addition of cuts, but not for the addition of variables needed in a branch-and-price method. Moreover, existing branch-cut-and-price frameworks (e.g., \cite{junger2000abacus,BolusaniEtal2024OO,gamrath2010experiments,sadykov2021bapcod}), remain less developed than the leading optimization solvers. In this sense, formulations that use only a fixed number of variables are easier to implement and can use the latest technology of commercial optimization solvers.

\subsection{\mytitle{Paper organization}}

We proceed by first establishing central definitions and notation in Section~\ref{section:preliminaries}. Section~\ref{section:corner} contains our main results, introducing a new way of generating Benders' cuts via projected corner relaxations.
Section~\ref{section:vrpsd} takes advantage of network flow structure and an explicit connection to DW reformulations,
then specializes the methodology to the 
two-stage capacitated vehicle routing problem with stochastic demands~\citep{gendreau201650th}, 
in which routes need to be decided in a first stage before customer demands are revealed, leading to second-stage recourse decisions.
The computational results, discussed in Section~\ref{section:experiments}, indicate that for instances with a large number of vehicles, corner Benders' cuts enhance the performance of a problem-specific, state-of-the-art branch-and-cut algorithm and compare favorably to the recently-proposed method by \cite{chen2024recovering}.

\section{\mytitle{Preliminaries: key notation and definitions}}
\label{section:preliminaries}

In this section, we establish the notation and briefly present the main mathematical tools that we use throughout the paper.

\subsection{\mytitle{Supports and Fenchel cuts}}
\label{subsection:fenchel}

The main goal of this work is to study valid inequalities for Problem~\eqref{problem_master}, for which a useful point of view is to regard such inequalities as Fenchel cuts~\citep{boyd1994fenchel}. 
The concept of a Fenchel cut is fundamental: whenever we are given an optimization oracle over a set, we can use this oracle to compute inequalities \emph{valid} (i.e., satisfied by all feasible points) for the set. In this way, Fenchel cuts arise as a natural consequence of fundamental results on the separation of convex sets~\citep{hiriart1996convex, boyd2004convex}, and led to the development of important tools for integer and stochastic programming, such as local cuts~\citep{chvatal2013local}, Lagrangian cuts~\citep{zou2019stochastic, chen2024recovering}, and Fenchel decomposition~\citep{ntaimo2013fenchel}. We start with the following definition.

\begin{definition}
    For $\mathcal{X} \subseteq \R^t$, the \textit{support} of $\mathcal{X}$ is the function~$\sigma_{\mathcal{X}} : \R^t \to \R \cup \{\pm \infty\}$ that sends each~$\alpha \in \R^t$ to the value~$\inf_{x \in \mathcal{X}} \{ \alpha^\T x \}$.
    \label{definition:support}
\end{definition}
\noindent
Notice that by the definition of the infimum,~$\sigma_{\mathcal{X}}(\alpha) = + \infty$ if~$\mathcal{X}$ is empty, and~$\sigma_{\mathcal{X}}(\alpha) = - \infty$ if~$\{\alpha^\T x \}_{x \in \mathcal{X}}$ is not bounded below by a finite real number.

We make a couple of additional observations based on Definition~\ref{definition:support}. First, the standard support function used in convex analysis~\citep{hiriart1996convex, cornuejols2006convex} replaces the infimum by the supremum in Definition~\ref{definition:support}; however, we shall see that for the purposes of writing Benders' cuts, it will be more convenient to use the infimum. Second, if each element of~$\mathcal{X}$ is identified with a tuple~$(x, y) \in \R^{t} \times \R^{t'}$, then instead of writing~$\sigma_{\mathcal{X}}((\alpha, \gamma))$ to refer to the value of~$\inf_{(x, y) \in \mathcal{X}} \{ \alpha^\T x + \gamma^\T y \}$, we use the shorthand~$\sigma_{\mathcal{X}}(\alpha, \gamma)$. 

With Definition~\ref{definition:support} in hand, obtaining valid inequalities for~$\mathcal{X}$ is immediate.
\begin{definition}
    For any set~$\mathcal{X} \subseteq \R^t$ and~$\alpha \in \R^t$ such that~$\sigma_{\mathcal{X}}(\alpha)$ is finite, an inequality (valid for~$\mathcal{X}$) is said to be a \textit{Fenchel cut} if it has the form~$\alpha^\T x \geq \sigma_{\mathcal{X}}(\alpha)$.
    \label{definition:fenchel}
\end{definition}

Definition~\ref{definition:fenchel} shows that, if we have access to an optimization oracle over~$\mathcal{X}$, then we can generate Fenchel cuts. As mentioned earlier, this is an old idea, and in the context of integer programming, it goes back to the seminal work of~\cite{boyd1994fenchel}, where the author generates cutting planes by calling an oracle that solves knapsack problems. In a similar spirit, we shall see in Section~\ref{section:corner} that Benders' cuts can be interpreted as Fenchel cuts that use an optimization oracle over the higher-dimensional polyhedron~$Y$.

\subsection{\mytitle{Value functions and epigraphs}}
\label{subsection:epigraph}

As discussed in Section~\ref{section:intro}, valid inequalities for Problem~\eqref{problem_master} can be derived using the concepts of value functions and epigraphs, which we formalize next.
Recall that $p$ denotes the number of rows in the linking constraints $Tx + Qy = h$.
We \emph{redefine} the value function $\fYx$ in terms of a generic set $\mathcal{Y}$ (which may differ from polyhedron $Y$),
a generalization that aids the discussion in Section~\ref{section:corner}.
For convenience, we also use a change of variable from $x \in \R^n$ to $w \in \R^p$, representing (for fixed $x$ values) the quantity $h - Tx$.
This enables us to simplify how we define the domain and epigraph of the value function in Definition~\ref{definition:domain_epigraph}.

\begin{definition}
The \textit{value function} with respect to a set~$\mathcal{Y} \subseteq \R^m$ is 
    $
        f_{\mathcal{Y}}(w) \defeq \inf_{y \in \mathcal{Y}} \{ d^\T y \suchthat Qy = w \},
    $
    for any $w \in \R^p$.
\label{definition:value_function}
\end{definition}

\begin{definition}
    Let~$\mathcal{Y} \subseteq \R^m$. The \textit{domain} and \textit{epigraph} of $\fYnoref$ are the sets
        $\dom(\fYnoref[\mathcal{Y}]) \defeq \{w \in \R^p \suchthat \fYnoref[\mathcal{Y}](w) < + \infty \}$
    and
        $\epi(\fYnoref[\mathcal{Y}]) \defeq \{ (w, \theta) \in \R^p \times \R \suchthat \theta \geq \fYnoref[\mathcal{Y}](w) \}$.
\label{definition:domain_epigraph}
\end{definition}

Let $\feasreg$ denote the feasible region of Problem~\eqref{problem_master}.
Since, for $x \in X$, 
    $\theta \ge \min_{y \in Y} \{d^\T y \suchthat Qy = h - Tx\}$
if and only if
    $\theta \ge \fY(h-Tx)$,
we may write
(applying Definition~\ref{definition:domain_epigraph} with~$\mathcal{Y} = Y$)
\begin{equation}
    \mathcal{F} = \{(x, \theta) \suchthat x \in X, (h - Tx, \theta) \in \epifY \}.
    \label{eq:feasible_region}\tag{\ensuremath{\mathcal{F}}}
\end{equation}
Note that if
    $\alpha^\T w + \alpha_0 \theta \geq \beta$
is valid for~$\epifY$, then
    $\alpha^\T (h - Tx) + \alpha_0 \theta \geq \beta$
is valid for $\feasreg$.

\subsection{\mytitle{Benders' cuts}}
\label{subsection:benders}

The following is our definition of a Benders' cut.
\begin{definition}
    An inequality~$\alpha^\T w + \alpha_0 \theta \geq \beta$ is a \textit{Benders' cut} if it is valid for~$\epifY$.
    \label{definition:benders_cut}
\end{definition}

Since variable~$\theta$ can increase arbitrarily in $\feasreg$
(for any $x \in X$, $(x,\theta) \in \feasreg$ for all $\theta \ge \fYx(x)$), 
we know that $\alpha_0 \geq 0$.
The use of a generic coefficient $\alpha_0$ in Definition~\ref{definition:benders_cut} allows us to treat optimality and feasibility cuts in a unified way (similarly to the approach taken in~\cite{fischetti2010note}): in the case that~$\alpha_0 = 0$, we have that the Benders' feasibility cut~$\alpha^\T w \geq \beta$ is a valid inequality for $\domfY$; otherwise, we may scale the inequality to get the optimality cut~$\theta \geq \frac{1}{\alpha_0} (-\alpha^\T w + \beta)$.

Before we close this section, we make a few more notes on Definition~\ref{definition:benders_cut}. First, by linear programming duality we know that the epigraph of the value function $\fY$ is convex and closed. (In fact, Proposition 1.50 of~\cite{mordukhovich2014easy} shows that~$\epi(f_{\mathcal{Y}})$ is convex and closed even for much more general choices of $\mathcal{Y} \subseteq \R^m$.) Therefore,~$\epifY$ can be written as the intersection of the halfspaces defined by a set of Benders' cuts. Second, to separate Benders' cuts, it suffices to consider the coefficients $(\alpha, \alpha_0)$ (ignoring~$\beta$), since, if $\alpha^\T w + \alpha_0 \theta \geq \beta$ is a Benders' cut, then~$\beta$ is at most~$\sigma_{\epifY}(\alpha, \alpha_0)$. In other words, we may restrict our attention only to Fenchel cuts for~$\epifY$.

\section{\mytitle{Corner Benders' cuts}}
\label{section:corner}

We are now ready to address how to use basis information to generate Benders' cuts more effectively, as illustrated in Figure~\ref{figure:epigraph}. Consider the scenario where we solve a relaxation of Problem~\eqref{problem_master} to obtain a vector~$(\bar{x}, \bar{\theta}) \in \R^{n + 1}$. Suppose that we solve the separation problem for~$\epifY$ (i.e., the \emph{Benders' subproblem}) and we obtain an inequality (valid for~$\epifY$) of the form
\begin{equation}
    \alpha^\T w + \alpha_0 \theta \geq \beta
    \label{ineq:benders_opt_cut}
\end{equation} 
such that $\alpha^\T (h - T \bar{x}) + \alpha_0 \bar{\theta} < \beta$. At this point, the standard Benders' decomposition method would add the cut~$\alpha^\T (h - Tx) + \alpha_0 \theta \geq \beta$ to the master problem and repeat the process by solving another relaxation of Problem~\eqref{problem_master}. 

Our main idea is to use the Benders' cut~\eqref{ineq:benders_opt_cut} to compute a set of inequalities that yield a tighter approximation of $\epifY$ than the single cut~\eqref{ineq:benders_opt_cut} alone.
To accomplish this, we first define in Section~\ref{subsection:corner_definitions} a conic relaxation~$C$ of polyhedron~$Y$ that we call a \textit{corner}. Next, we show in Section~\ref{subsection:corner_benders} how to find a corner~$C$ such that every point in the epigraph of~$\fY[C]$ satisfies the given Benders' cut~\eqref{ineq:benders_opt_cut} (see Figure~\ref{figure:epigraph}). Rather than adding only a single optimality cut to the master problem, we propose in Section~\ref{subsection:corner_separation} to use an algorithm that can efficiently separate multiple facets of~$\epifY[C]$. Finally, in Section~\ref{subsection:corner_bounds} we establish the existence of a single corner~$C$ such that, in replacing~$\fY$ by~$\fY[C]$ in Problem~\eqref{problem_master}, we still preserve its optimal value~$z^*$.

\subsection{\mytitle{Definitions}}
\label{subsection:corner_definitions}

Henceforth,~$\cone(\cdot)$ is the conical-hull operator, and the summation of sets refers to the Minkowski sum.
We first define a corner of $Y$ as any conical relaxation whose apex lies at an extreme point of $Y$.
\begin{definition}
    A set~$C \subseteq \R^m$ is a \emph{corner} of~$Y$, if~$C = \{y^*\} + \cone(R)$ contains~$Y$, where~$y^* \in \R^m$ is an extreme point of~$Y$ and~$R \subseteq \R^m$. If~$C \subseteq \R^m$ is a corner of~$Y$ and~$\sigma_{C}(\gamma)$ is finite, then~$C$ is said to be an \emph{optimal corner} with respect to~$\sigma_Y(\gamma)$.
    \label{definition:corner}
\end{definition}

The next result justifies the term ``optimal corner''.
\begin{lemma}
If~$C \subseteq \R^m$ is an optimal corner with respect to~$\sigma_Y(\gamma)$, then~$\sigma_C(\gamma) = \sigma_Y(\gamma)$.
\label{lemma:corner_equal_support}
\end{lemma}
\begin{proof}
By the definition of a corner,~$C = \{y^*\} + \cone(R)$, where~$y^*$ is an extreme point of~$Y$ and~$R \subseteq \R^m$. For every~$\bar{y} \in C \supseteq Y$, we can write~$\bar{y} = y^* + \sum_{r \in R} \mu_r \, r$, where~$\mu_r \geq 0$, for all~$r \in R$. Hence,~$\gamma^\T \bar{y} = \gamma^\T y^* + \sum_{r \in R} \mu_r (\gamma^\T r) \geq \gamma^\T y^*$, where the last inequality follows because~$\sigma_C(\gamma)$ is finite, so~$\gamma^\T r \geq 0$, for every~$r \in R$. Since~$y^*$ is a point in~$Y \subseteq C$ that attains the lower bound of~$\gamma^\T y^*$, we conclude that~$\sigma_C(\gamma) = \sigma_Y(\gamma) = \gamma^\T y^*$.
\end{proof}

In Remark~\ref{remark:simplex}, we illustrate how we will construct corners in our experiments.

\begin{remark}
Suppose~$Y = \{y \in \R^m : A y = b, y \geq 0\}$ is a polyhedron in standard equality form. For each~$j \in [m]$, we use~$A_j$ to denote the~$j$-th column of~$A$ (so~$Ay = \sum_{j \in [m]} A_j y_j$) and for a set~$J \subseteq [m]$, we use~$A_J$ to denote the submatrix of~$A$ containing exactly the set of columns~$\{A_j\}_{j \in J}$. Suppose that~$B$ is an optimal basis for the problem of minimizing~$\gamma^\T y$ over~$Y$. Let~$N = [m] \setminus B$ and rewrite~$A y = b$ as the system~$A_B  y_B + A_N  y_N = b$. Define~$\bar{b} \defeq (A_B)^{-1} b$ and~$\bar{a}^j \defeq (A_B)^{-1} A_j$, for each~$j \in N$. Thus, we rewrite~$Ay = b$ as~$y_B = \bar{b} - \sum_{j \in N} \bar{a}^j  y_j$. By the choice of~$B$,~$y^* = (y^*_B, y^*_N) = (\bar{b}, \mathbf{0})$ is an optimal basic feasible solution. Each nonbasic variable~$j \in N$ gives a ray~$r^j \in \R^m$ with entries
$$
r^j_i =
\begin{cases}
-\bar{a}^j_i, & \text{if } i \in B, \\
1, & \text{if } i = j, \\
0, & \text{otherwise.} \\
\end{cases}
$$
Since each nonbasic variable has a nonnegative reduced cost, the set~$C = \{y^*\} + \cone(\{r^j\}_{j \in N})$ is an optimal corner with respect to~$\sigma_{Y}(\gamma)$. \hfill \Halmos
\label{remark:simplex}
\end{remark}

Now that we have the definition of corners, we formalize what we mean by corner Benders' cuts.
\begin{definition}
    An inequality~$\alpha^\T w + \alpha_0 \theta \geq \beta$ is a \emph{corner Benders' cut} if it is valid for~$\epifY[C]$, where~$C$ is a corner of~$Y$.
    \label{definition:corner_benders_cut}
\end{definition}

\noindent
Note that since~$C$ is a relaxation of~$Y$, it follows from the definition of value functions that~$\epifY[C]$ contains~$\epifY$; in other words, corner Benders' cuts are also Benders' cuts (i.e., they are valid inequalities for the epigraph of~$\fY$).

\subsection{\mytitle{Finding an optimal corner from a Benders' cut}}
\label{subsection:corner_benders}

Our strategy will be to find an optimal corner based on a Benders' cut.
To accomplish this, we prove Lemma~\ref{lemma:support}, which shows that
depending on the choice of the objective function, optimizing over a set~$\mathcal{Y} \subseteq \R^m$ is equivalent to optimizing over the epigraph~$\epifY[\mathcal{Y}] \subseteq \R^{p + 1}$. 

\begin{lemma}
    Let~$\mathcal{Y} \subseteq \R^m$. For every~$(\alpha, \alpha_0) \in \R^p \times \R_+$, it holds that~$$\sigma_{\mathcal{Y}}(Q^\T \alpha + \alpha_0 d) =  \sigma_{\epifY[\mathcal{Y}]} (\alpha, \alpha_0).$$
    \label{lemma:support}
\end{lemma}
\begin{proof}
Since~$\alpha_0 \geq 0$, we may write
\begin{align}
    \sigma_{\epifY[\mathcal{Y}]}(\alpha, \alpha_0) 
        &= \inf_{(w,\theta) \in \epifY[\mathcal{Y}]} \left\{ \alpha^\T w + \alpha_0 \theta \right\} \nonumber
        \\
        &= \inf_{w,\theta} \left\{ \alpha^\T w + \alpha_0 \theta \suchthat \theta \ge \inf_{y \in \mathcal{Y}} \{ d^\T y \suchthat Qy = w \} \right\} \nonumber
        \\
        &= \inf_{w} \left\{ \alpha^\T w + \alpha_0 \inf_{y \in \mathcal{Y}} \{ d^\T y \suchthat Qy = w \} \right\} \label{support_proof_eq1}
        \\
        &= \inf_{w, y} \left\{ \alpha^\T w + \alpha_0 (d^\T y) \suchthat Qy = w, y \in \mathcal{Y} \right\} \label{support_proof_eq2}
        \\
        &= \inf_{y \in \mathcal{Y}} \{\alpha^\T (Qy) + \alpha_0 (d^\T y) \} \nonumber
        \\
        &= \sigma_{\mathcal{Y}}(Q^\T \alpha + \alpha_0 d). \nonumber
\end{align}
\noindent
The fourth equality holds because optimizing first over~$w$ and then over~$y$ is equivalent to jointly optimizing over~$(w, y)$ (see Section 4.1.3, ``Optimizing over some variables'', in~\cite{boyd2004convex}). More explicitly, note that if we fix~$w$ to some arbitrary~$\bar{w} \in \dom(\fY[\mathcal{Y}])$, both~\eqref{support_proof_eq1} and~\eqref{support_proof_eq2} reduce to~$\alpha^\T \bar{w} + \alpha_0 \inf_{y \in \mathcal{Y}} \{ d^\T y \suchthat Qy = \bar{w} \} = \alpha^\T \bar{w} + \alpha_0 \fY[\mathcal{Y}](\bar{w})$.
\end{proof}

The condition that~$\alpha_0 \geq 0$ in Lemma~\ref{lemma:support} is required, as otherwise we might have that~$\sigma_{\mathcal{Y}}(Q^\T \alpha + \alpha_0  d)$ is finite while~$\sigma_{\epifY[\mathcal{Y}]} (\alpha, \alpha_0) = - \infty$. 

By Lemma~\ref{lemma:corner_equal_support}, we know that an optimal corner is a translated cone~$C$ that is a relaxation of polyhedron~$Y$
for which optimizing a given linear objective function attains the same optimal value.
The next theorem tells us that the epigraphs of~$\fY$ and~$\fY[C]$ also have a similar property. Although this result may seem straightforward, it is crucial to our approach, since it shows that, given a Benders' cut~$\alpha^\T w + \alpha_0 \theta \geq \beta$, one can simply optimize over~$Y$ to find a corner~$C$ such that~$\alpha^\T w + \alpha_0 \theta \geq \beta$ is valid for~$\epifY[C]$.

\begin{theorem}
    Let~$C$ be a corner of~$Y$. 
    A Benders' cut $\alpha^\T w + \alpha_0 \theta \geq \beta$ is valid for~$\epifY[C]$ if and only if~$C$ is an optimal corner with respect to~$\sigma_{Y} (Q^\T \alpha + \alpha_0 d)$. Moreover, there exists a finite set~$\mathcal{C}$ of corners of~$Y$ such that~$\epifY = \bigcap_{C' \in \mathcal{C}} \epifY[C']$.
    \label{thm:benders_corner}
\end{theorem}
\begin{proof}
    Fix a Benders' cut~$\alpha^\T w + \alpha_0 \theta \geq \beta$ and note that, as we observed in Section~\ref{subsection:benders},~$\alpha_0 \geq 0$ and~$\beta \leq \sigma_{\epifY}(\alpha, \alpha_0)$. 
    If~$\alpha^\T w + \alpha_0 \theta \geq \beta$ is valid for~$\epifY[C]$, we have that
        $-\infty < \beta \leq \sigma_{\epifY[C]}(\alpha, \alpha_0) < +\infty$,
    where the last inequality follows from $C$ being nonempty. 
    Hence, by Lemma~\ref{lemma:support}, $\sigma_C(Q^\T \alpha + \alpha_0 d)$ is also finite. So, by Definition~\ref{definition:corner}, the set~$C$ is an optimal corner with respect to~$\sigma_Y(Q^\T \alpha + \alpha_0 d)$.
    
    Conversely, let~$C \subseteq \R^m$ be an optimal corner with respect to~$\sigma_{Y}(Q^\T \alpha + \alpha_0 d)$. Then
    \begin{equation*}
        \sigma_{\epifY}(\alpha, \alpha_0) \overset{\text{Lemma~\ref{lemma:support}}}{=} \sigma_Y(Q^\T \alpha + \alpha_0 d) \overset{\text{Lemma~\ref{lemma:corner_equal_support}}}{=} \sigma_{C}(Q^\T \alpha + \alpha_0 d) \overset{\text{Lemma~\ref{lemma:support}}}{=} \sigma_{\epifY[C]}(\alpha, \alpha_0),
    \end{equation*}
    which proves that~$\beta \leq \sigma_{\epifY[C]}(\alpha, \alpha_0)$, so~$\alpha^\T w + \alpha_0 \theta \geq \beta$ is valid for~$\epifY[C]$.
    
    To prove the second part of the statement, let~$\mathcal{C}$ be the set of corners associated with every feasible basis of~$Y$ (see, for example, Remark~\ref{remark:simplex}). It is clear that~$\mathcal{C}$ is finite. Moreover, for every~$C' \in \mathcal{C}$, we have~$\epifY \subseteq \epifY[C']$. Indeed, since~$Y \subseteq C'$, for any~$(w,\theta)\in \epifY$, Definition~\ref{definition:value_function} implies that~$\theta \geq \fY(w) \geq \fY[C'](w)$, meaning that~$(w,\theta)\in \epifY[C']$. This argument shows that~$\epifY \subseteq \bigcap_{C' \in \mathcal{C}} \epifY[C']$. But we have just proven that any valid inequality~$\alpha^\T w + \alpha_0 \theta \geq \beta$ for $\epifY$ is also valid for~$\epifY[C']$, where~$C'$ is an optimal corner with respect to~$\sigma_{Y}(Q^\T \alpha + \alpha_0 d)$. In particular, by choosing the basis appropriately, we may assume that~$C'$ belongs to the set~$\mathcal{C}$, proving the reverse inclusion.
\end{proof}

\subsection{\mytitle{Separating corner Benders' cuts via polarity}}
\label{subsection:corner_separation}

Throughout this subsection, we fix a corner~$C = \{y^*\} + \cone(R)$, and we consider the problem of generating strong valid inequalities for~$\epifY[C]$. Recall that an \textit{inner description} (also called~$\mathcal{V}$-polyhedral, see~\cite{ziegler2012lectures} and~\cite{prlp}) of a polyhedral set is given by the Minkowski sum of the convex hull of a set of points with the conical hull of a set of rays. Since we have access to an inner description of~$C$, we immediately obtain an inner description of~$\epifY[C]$.

\begin{lemma}
    Let~$C = \{y^*\} + \cone(R)$ be a corner of~$Y$ with~$R$ finite and let~$(w^*, \theta^*) = (Qy^*, d^\T y^*)$. Then~$$\epifY[C] = \{(w^*, \theta^*)\} + \cone(\{(Qr, d^\T r)\}_{r \in R} \cup \{(\mathbf{0}, 1)\}).$$
    \label{lemma:epigraph_cone}
\end{lemma}
\begin{proof}
    By the inner description of~$C$ and the definition of the epigraph of~$\fY[C]$,
    \begin{align*}
        \epifY[C] = & \{(w, \theta) \suchthat \theta \geq \inf_{y \in C} \{d^\T y : Qy = w\} \} \\
        = & \{(w, \theta) \suchthat \theta \geq \min_{y \in C} \{d^\T y : Qy = w\}, w \in \dom(\fY[C]) \} \\
        = & \{(Qy, d^\T y) + \mu_0 (\mathbf{0}, 1) \suchthat y \in C, \mu_0 \geq 0\} \\
        = & \left\{ \left(Q y^* + \sum_{r \in R} \mu_r (Qr), d^\T y^* + \mu_0 + \sum_{r \in R} \mu_r (d^\T r) \right) \suchthat \mu_i \geq 0, ~\forall i \in R\cup\{0\} \right\},
    \end{align*}
    where the second equality holds since~$R$ is finite, so for every~$w \in \dom(\fY[C])$, the set~$\{y \in C : Qy = w\}$ is polyhedral.
\end{proof}

The finiteness of~$R$ in Lemma~\ref{lemma:epigraph_cone} is required, since otherwise we may consider the following example. Assume that~$(w^*, \theta^*) = (0, 0) \in \R^2$ and~$R = \{(y_1, y_2, y_1 / y_2)\}_{y_1, y_2 \in [1, +\infty)}$, with~$Qr = y_1$ and~$d^\T r = y_1 / y_2$, for each~$r = (y_1, y_2, y_1 / y_2) \in R$. For any~$\bar{y}_1 \in [1, +\infty)$, we have~$\inf_{y \in C} \{d^\T y : Qy = \bar{y}_1 \} = 0$, so~$(\bar{y}_1, 0) \in \epifY[C]$. However, since~$d^\T r > 0$ for every~$r \in R$, the point~$(\bar{y}_1, 0)$ does not belong to~$\{(w^*, \theta^*)\} + \cone(\{(Qr, d^\T r)\}_{r \in R} \cup \{(\mathbf{0}, 1)\})$. 

Suppose now that we are given a candidate point~$(\bar{w}, \bar{\theta})$ that we wish to separate from~$\epifY[C]$. With Lemma~\ref{lemma:epigraph_cone} in hand, we can do this by searching for an infeasibility certificate for the system~
\begin{equation}
    \min_{\mu \geq 0} \left\{0 \suchthat \sum_{r \in R} (Qr) \mu_r = \bar{w},~\sum_{r \in R} (d^\T r) \mu_r \leq \bar{\theta}\right\}.
    \label{problem:cglp_simple}
\end{equation}
(In Appendix~\ref{appendix:corner_restricted_dual}, we show a connection between this approach and the method of~\cite{chen2022generating} to generate Lagrangian cuts by optimizing over restricted dual problems.) Notice, however, that whenever~\eqref{problem:cglp_simple} is infeasible, its dual is unbounded, and therefore one would typically have to choose among different normalizations~\citep{balas1993lift, louveaux2015strength, serra2020reformulating}. Here we propose to separate~$(\bar{w}, \bar{\theta})$ from~$\epifY[C]$ by normalizing the cut-generating linear program through known ideas from \textit{polarity}~\citep{Balas79, schrijver1998theory, cadoux2013reflections}.

To do so, we first translate~$\epifY[C]$ so that the origin lies in the relative interior of the set. In other words, we take a point~$(w', \theta')$ in the relative interior of~$\epifY[C]$ (this can be found by projecting a point in the relative interior of~$Y$) and we consider the set~$$E \defeq \epifY[C] - (w', \theta') = \{(w^*, \theta^*) - (w', \theta')\} + \cone(\{(Qr, d^\T r)\}_{r \in R} \cup \{(\mathbf{0}, 1)\}),$$ where~$w^* = Qy^*$ and~$\theta^* = d^\T y^*$. We then consider the following particular case of a \textit{reverse polar set of~$E$} (see~\cite{Balas79}):
\begin{equation}
    E^{\#} \defeq
    \left\{ (\alpha, \alpha_0) \in \R^{p + 1} \suchthat~
    \begin{aligned}
    & \alpha^\T (w^* - w') + \alpha_0  (\theta^* - \theta') \geq -1, & \\
    & \alpha^\T (Qr) + \alpha_0  (d^\T r) \geq 0, & \forall r \in R \\
    & \alpha_0 \geq 0,
    \end{aligned}
    \right\},
    \label{epigraph_rev_polar}
\end{equation}
and we refer to the inequalities associated with vectors in~$R$ as \textit{ray inequalities} of~$\revpolar$. 

Observe that, as pointed out by~\cite{Balas79}, the standard polar set of~$E$ is given by~$E^\circ = - \revpolar$ (see, for instance, Theorem 9.1 of~\cite{schrijver1998theory}). We use the reverse polar~$\revpolar$ simply to be consistent with our choice of expressing Benders' cuts in the form~$\alpha^\T w + \alpha_0 \theta \geq \beta$ rather than~$\alpha^\T w + \alpha_0 \theta \leq \beta$.

Let~$\alpha^\T w + \alpha_0 \theta \geq \beta$ be a valid inequality for~$E$ that is not an implicit equality. Since the origin lies in the relative interior of~$E$, we know that~$\beta < 0$ and by scaling, we may assume that~$\beta = -1$. It can now be easily verified that~$(\alpha, \alpha_0)$ belongs to~$\revpolar$. In fact, when~$E$ is full dimensional, it is known that~$\revpolar$ is bounded, with its extreme points corresponding to facets of~$E$ (see~\cite{schrijver1998theory} and~\cite{cadoux2013reflections}), which in turn translate into facets of~$\epifY[C]$. In our case, however, the set~$E$ is not necessarily full dimensional, and we need to consider the \textit{recession cone} of~$\revpolar$, denoted~$\recc(\revpolar)$. Given our use of the reverse polar~$\revpolar$ and this additional technical detail on its boundedness, we state Theorem~\ref{thm:polar} below. (This result follows from standard arguments on polarity, and we leave the proof to Appendix~\ref{appendix:proof_polar}.)

\begin{restatable}{theorem}{polar}
    Let~$C = \{y^*\} + \cone(R)$ be a corner of~$Y$.
    Let $(w^*, \theta^*) = (Qy^*, d^\T y^*)$,
    $(w', \theta')$ be a point in the relative interior of~$\epifY[C]$,
    and~$(\bar{w}, \bar{\theta})$ be an arbitrary point in~$\R^p \times \R$. 
    Define~$\revpolar$ as in~\eqref{epigraph_rev_polar} and consider the optimization problem
    \begin{equation}
        z^\circ \defeq \min_{\alpha,\alpha_0} \left\{ \alpha^\T (\bar{w} - w') + \alpha_0 (\bar{\theta} - \theta') : (\alpha, \alpha_0) \in \revpolar \right\}.
        \label{problem_polar}
    \end{equation}
    It holds that
    \begin{enumerate}[(a)]
        \item if~$z^\circ \geq -1$, then~$(\bar{w}, \bar{\theta})$ belongs to~$\epifY[C]$;
        \item if~$z^\circ = - \infty$ (i.e., Problem~\eqref{problem_polar} is unbounded), then we have a vector~$(\alpha, \alpha_0) \in \recc(\revpolar)$ such that~$\alpha^\T w + \alpha_0  \theta = \alpha^\T w' + \alpha_0  \theta'$ defines an implicit equality of~$\epifY[C]$ that is violated by~$(\bar{w}, \bar{\theta})$;
        \item otherwise,~$- \infty < z^\circ < -1$ and we have an extreme point~$(\alpha, \alpha_0)$ of~$\revpolar$ such that the inequality~$\alpha^\T w + \alpha_0 \theta \geq -1 + \alpha^\T w' + \alpha_0 \theta'$ defines a facet of~$\epifY[C]$ that is violated by~$(\bar{w}, \bar{\theta})$.
    \end{enumerate}
    \label{thm:polar}
\end{restatable}

\paragraph{\textbf{\mytitle{Row generation of ray inequalities.}}}
We solve \eqref{problem_polar} via Algorithm~\ref{algorithm:solve_polar} (\solvereversepolar).
We adopt a row generation approach for the ray inequalities in $\revpolar$:
in each iteration, we solve a relaxed variant of \eqref{problem_polar}, which we call~$(\tsc{CGLP})$ in the algorithm, over constraints for only a partial set of rays.
This is because, due to the projection~$r \to (Qr, d^\T r)$, several of the ray inequalities in~$\revpolar$ may be redundant. For example, if~$r^1, r^2 \in R$ satisfy~$Qr^1 = Qr^2$ and~$0 < d^\T r^1 < d^\T r^2$, then~$(Qr^2, d^\T r^2)$ is not an extreme ray of~$\epifY[C]$. Consequently, the ray inequality for~$r^2$ in~$\revpolar$ is redundant. 

The $(\tsc{CGLP})$ that we will solve contains ray inequalities for two sets of rays: ``current'' rays $R'$ and ``warm-start'' rays $\tilde{R}$. In each iteration, the set~$R'$ will be augmented with rays from a set~$R''$ (line~\ref{alg_solve_polar:choose_rays} of \solvereversepolar), which contains the rays from~$R$ that correspond to violated ray inequalities with respect to the current solution~$(\bar{\alpha}, \bar{\alpha}_0)$. There are various strategies for selecting rays $R'$ from~$R''$ based on which ray inequalities we want to impose for Problem~\cglp. For example, one could choose to add a fixed number of inequalities with the maximum violation. We tailor this to our specific formulation of the vehicle routing problem with stochastic demands (see Section~\ref{subsection:implementation} for further details).

\begin{algorithm}
\textbf{Input:} Candidate solution~$(\bar{x}, \bar{\theta}) \in \R^{n + 1}$; point~$(w', \theta')$ in the relative interior of~$\epi(f_Y)$; a corner~$C = \{y^*\} + \cone(R)$ given as point~$y^* \in \R^p$ and finite set~$R \subseteq \R^p$; and set~$\tilde{R} \subseteq R$ of ``warm-start'' rays. \\
\textbf{Output:} Optimal value~$z^\circ$ of Theorem~\ref{thm:polar}; tuple~$(\rho, \rho_0, \beta)$ (representing a corner Benders' cut of the form~$\rho^\top w + \rho_0 \theta \geq \beta$); and updated set of ``warm-start'' rays.
\begin{algorithmic}[1]
\Procedure {SolveReversePolar}{$(\bar{x}, \bar{\theta}), (w', \theta'), y^*, R, \tilde{R}$}
    \State {Set~$(w^*, \theta^*) = (Q y^*, d^\top y^*)$ and~$\bar{w} = h - T \bar{x}$.}
    \State {\multiline{Set~$R' = \emptyset$ and start Problem~$\tsc{(CGLP)}$ with the linear program below.
    \begin{equation}
        \begin{aligned}
        \tsc{(CGLP)}~~~~~~z^\circ ~=~ & \min_{\alpha,\alpha_0} && \alpha^\top (\bar{w} - w') + \alpha_0 (\bar{\theta} - \theta') && \\
        & \textsuchthat{} && \alpha^\top (w^* - w') + \alpha_0  (\theta^* - \theta') \geq -1, &&~~~~(\xi) \\
        &&& \alpha^\top (Qr) + \alpha_0  (d^\top r) \geq 0, & \forall r \in R' \cup \tilde{R}, &~~~~(\mu_r) \\
        &&& \alpha_0 \geq 0. &&~~~~(\mu_0)
        \end{aligned}
        \nonumber
    \end{equation}} \label{alg_solve_polar:cglp}}
    \State {$(\bar{\alpha}, \bar{\alpha}_0) \gets (\mathbf{0}, 0)$}
    \Repeat
        \State {$R'' \gets \emptyset$}
        \State {Solve Problem~\cglp to get~$z^\circ$ and update~$(\bar{\alpha}, \bar{\alpha}_0)$ to the obtained extreme point/ray.}
        \For {$r \in R$} \label{alg_solve_polar:begin_for}
            \If {$\bar{\alpha}^\top (Qr) + \bar{\alpha}_0 (d^\top r) < 0$}
                \State {$R'' \gets R'' \cup \{r\}$}
            \EndIf
        \EndFor \label{alg_solve_polar:end_for}

        \If {$R'' \neq \emptyset$}
            \State \multiline{Choose some rays~$r \in R''$ and add them to~$R'$ (effectively adding the corresponding inequalities~$\alpha^\top (Qr) + \alpha_0 (d^\top r) \geq 0$ to Problem~\cglp). \label{alg_solve_polar:choose_rays} }
        \EndIf
    \Until{$R'' = \emptyset$}
    \If{$z^\circ \geq -1$}
        \State{\textbf{return} $((\mathbf{0}, 0, 0), z^\circ, \tilde{R})$.}
    \ElsIf{$z^\circ = -\infty$}
        \State{\textbf{return} $((\bar{\alpha}, \bar{\alpha}_0, \beta = \bar{\alpha}^\top w' + \bar{\alpha}_0 \theta'), z^\circ, \tilde{R})$.}
    \Else
        \State \multiline{Let~$(\bar{\xi}, \bar{\mu})$ be optimal dual multipliers associated with the current optimal solution~$(\bar{\alpha}, \bar{\alpha}_0)$.}
        \State {$\tilde{R}' \gets \{r \in R' : \bar{\mu}_r > 0 \}$. \label{alg_solve_polar:warm_start_cuts}} 
        \State{\textbf{return} $((\bar{\alpha}, \bar{\alpha}_0, \beta = -1 + \bar{\alpha}^\top w' + \bar{\alpha}_0 \theta'), z^\circ, \tilde{R} \cup \tilde{R}')$.}
    \EndIf
    \EndProcedure
\end{algorithmic}
\caption{\textsc{SolveReversePolar}}
\label{algorithm:solve_polar}
\end{algorithm}

The ``warm-start'' ray set~$\tilde{R}$ is passed as an argument to \solvereversepolar, which is repeatedly called as a subroutine in Algorithm~\ref{algorithm:separate_corner} (\separatecorner).
The procedure \separatecorner begins by using Theorem~\ref{thm:benders_corner} to build an optimal corner~$C$ from a Benders' cut~$\alpha^\T w + \alpha_0 \theta \geq \sigma_{\epifY}(\alpha, \alpha_0)$.
Then, in line~\ref{alg_separate_cuts:call-subroutine-solve-reverse-polar}, \solvereversepolar is called to find violated corner Benders' cuts with respect to~$\epifY[C]$.
The ``warm-start'' rays $\tilde{R}$ initialize Problem~\cglp with ray inequalities that were active for cuts generated in prior iterations, collected in \solvereversepolar at line~\ref{alg_solve_polar:warm_start_cuts}.

\begin{algorithm}
\textbf{Input:} Set~$\tilde{\mathcal{F}}$ that is a relaxation of the feasible region~$\mathcal{F}$; point~$y'$ that belongs to the relative interior of~$Y$; and vector~$(\alpha, \alpha_0) \in \R^{p + 1}$ (representing a Benders' cut of the form~$\alpha^\T w + \alpha_0 \theta \geq \sigma_{\epi(f_Y)}(\alpha, \alpha_0)$). \\
\textbf{Output:} Corner~$C$ of~$Y$ and finite set of vectors~$\Omega \subseteq \R^{p + 2}$ (where each element~$(\rho, \rho_0, \beta) \in \Omega$, represents a corner Benders' cut of the form~$\rho^\T w + \rho_0 \theta \geq \beta$).
\begin{algorithmic}[1]
\Procedure {SeparateCornerBendersCuts}{$\tilde{\mathcal{F}}, y', \alpha, \alpha_0$}
\State \multiline{Set~$C = \{y^*\} + \cone(R)$ to be an optimal corner w.r.t.~$\sigma_Y(Q^\T \alpha + \alpha_0 d)$ (for example, as described in Remark~\ref{remark:simplex}).}
\State {Set~$(w', \theta') = (Q y', d^\T y' + \varepsilon)$, where~$\varepsilon > 0$ is an arbitrary constant.}
\State {$\Omega \gets \emptyset$}
\State {$\tilde{\mathcal{R}} \gets \emptyset$}
\Repeat
    \State {\multiline{Solve problem~$\tsc{(M)}$ below to get a candidate solution~$(\bar{x}, \bar{\theta})$.
    \begin{equation}
        \begin{aligned}
        \tsc{(M)}~~~~& \min_{x,\theta} && c^\T x + \theta & \\
        & && (x, \theta) \in \tilde{\mathcal{F}}, & \\
        &&& \rho^\T (h - Tx) + \rho_0 \theta \geq \beta,~~~& \forall (\rho, \rho_0, \beta) \in \Omega.
        \end{aligned}
        \nonumber
    \end{equation}}}
    \State \label{alg_separate_cuts:call-subroutine-solve-reverse-polar} {$((\rho, \rho_0, \beta), z^\circ, \tilde{R}) \gets \solvereversepolar((\bar{x}, \bar{\theta}), (w', \theta'), y^*, R, \tilde{R})$}
    \If {$z^\circ = -\infty$}
        \State {$\mathcal{C} \gets \mathcal{C} \cup \{(\rho, \rho_0, \beta), (-\rho, -\rho_0, -\beta)\}$}
    \ElsIf {$z^\circ < -1$}
        \State {$\Omega \gets \Omega \cup \{(\rho, \rho_0, \beta)\}$}
    \EndIf
\Until {$z^\circ \geq -1$}
\State {\textbf{return} $\Omega$}
\EndProcedure
\end{algorithmic}
\caption{\textsc{SeparateCornerBendersCuts}}
\label{algorithm:separate_corner}
\end{algorithm}

While we cannot guarantee that the rays in the ``warm-start'' set~$\tilde{R}$ are extreme, we next show that, in a sense, at least these rays belong to the boundary of~$\epifY[C]$.

\begin{claim}
    Let~$((\bar{\alpha}, \bar{\alpha}_0, \beta), z^\circ, \tilde{R} \cup \tilde{R}')$ be the output of \solvereversepolar (with~$\tilde{R}'$ being set as in line~\ref{alg_solve_polar:warm_start_cuts} of Algorithm~\ref{algorithm:solve_polar}). Suppose that~$-\infty < z^\circ < -1$ and consider the facet~$F = \epifY[C] \cap \{(w, \theta) : \bar{\alpha}^\T w + \bar{\alpha}_0 \theta = \beta\}$. Then, for every~$r \in \tilde{R}'$,~$(w^* + Q r, \theta^* + d^\T r)$ belongs to~$F$.
\end{claim}
\begin{proof}
Fix~$r \in \tilde{R}'$. Since~$(Qr, d^\T r)$ is a recession direction of~$\epifY[C]$, we know that~$(w^* + Q r, \theta^* + d^\T r)$ belongs to~$\epifY[C]$. Since~$\bar{\mu}_r > 0$, we have by complementary slackness that~$\bar{\alpha}^\T (Qr) + \bar{\alpha}_0 (d^\T r) = 0$. Combining this with the fact that~$(w^*, \theta^*) \in F$, we conclude that~$\bar{\alpha}^\T (w^* + Qr) + \bar{\alpha}_0 (\theta^* + d^\T r) = \beta$.
\end{proof}

\subsection{\mytitle{Finding a corner that recovers the optimal bound}}
\label{subsection:corner_bounds}

The previous subsections contain the main message of this paper: given a Benders' cut, we find an optimal corner~$C$ using Theorem~\ref{thm:benders_corner} and we then separate corner Benders' cuts using Theorem~\ref{thm:polar} and Algorithms~\ref{algorithm:separate_corner} and~\ref{algorithm:solve_polar}. The goal of this subsection is to show that by generating corner Benders' cuts for a fixed~$\epifY[C]$, we can already recover the optimal bound~$z^*$ in Problem~\eqref{problem_master}. More precisely, we prove the following result.
\begin{proposition}
Let~$z^*$ be the optimal value of Problem~\eqref{problem_master} and let~$\hat{\alpha}$ be an optimal solution to
\begin{equation}
    \max_{\alpha \in \R^p} \left\{ - \alpha^ \T h + \sigma_{X}(c + T^ \T \alpha) + \sigma_{Y}(Q^\T \alpha + d)\right\}.
    \label{problem:lagrangian1}
\end{equation}
If~$C$ is an optimal corner with respect to~$\sigma_Y(Q^\T \hat{\alpha} + d)$, then it holds that
\begin{equation}
    \min_{x, \theta} \left\{c^\T x + \theta \suchthat x \in X, (h - Tx, \theta) \in \epifY[C] \right\} = z^*.
    \label{problem:corner_bound}
\end{equation}
\label{prop:strong_corner}
\end{proposition}

Before we dive into the proof of Proposition~\ref{prop:strong_corner}, let us discuss the relevance of this result and its connections to the existing literature. It is immediate that the \textit{objective function cut}~$c^\T x + \theta \geq z^*$ recovers the optimal bound~$z^*$. Recall, however, that our overall goal is to accelerate branch-and-bound MIP solvers that repeatedly solve relaxations that take the form of the linear program shown in Problem~\eqref{problem_master}. In this sense, it is often detrimental to a MIP solver to impose the objective function cut at the root node of a branch-and-bound tree~\citep{gamrath2010experiments}. 
\cite{chen2024recovering} argue that an explanation for this poorer performance is that adding the objective function cut to the LP relaxation usually leads to an optimal face with almost the same dimension as the original polyhedron, leading to performance degradation due to dual degeneracy. In contrast, Proposition~\ref{prop:strong_corner} and Algorithm~\ref{algorithm:separate_corner} allow us to decompose the objective function cut into multiple corner Benders' cuts; and this might lead to lower-dimensional optimal faces. Indeed, as we later show in our computational experiments, when taking Problem~\eqref{problem_master} to be the root node of a branch-and-bound tree, the approach based on Proposition~\ref{prop:strong_corner} outperforms both the objective function cut and a single Benders' cut method inspired by the work of \cite{chen2024recovering} (see Theorem~\ref{thm:lagrangian_cut} next).

Let us now address the proof of Proposition~\ref{prop:strong_corner}. First, we show that, perhaps not surprisingly, separating a Fenchel cut for~$\feasreg$ (the feasible region of Problem~\eqref{problem_master}) amounts to separating a Fenchel cut for~$X$ and another Fenchel cut for~$\epifY$. (In Appendix~\ref{appendix:block_diagonal}, we further show that Theorem~\ref{thm:lagrangian_cut} can be easily adapted to the case where Problem~\eqref{problem:basic} has a ``block-diagonal'' structure.)

\begin{theorem}
    Let~$(\rho, \rho_0) \in \R^n \times \R_+$ and assume~$\sigma_{\feasreg}(\rho, \rho_0)$ is finite. Let~$\hat{\alpha} \in \R^p$ be optimal for the problem
    \begin{equation}
        \max_{\alpha \in \R^p} \left\{ - \alpha^ \T h + \sigma_{X}(\rho + T^ \T \alpha) + \sigma_{Y}(Q^\T \alpha + \rho_0  d)\right\}.
        \label{problem:lagrangian2}
    \end{equation}
    Then the following hold.
    \begin{enumerate}[(i)]
        \item~$(\rho + T^ \T \hat{\alpha})^\T x \geq \sigma_{X}(\rho + T^ \T \hat{\alpha})$ is valid for~$X$; \label{item:langrangian1}
        \item~$\hat{\alpha}^\T (h - Tx) + \rho_0  \theta \geq \sigma_{\epifY} (\hat{\alpha}, \rho_0)$ is valid for~$\epifY$; and \label{item:langrangian2}
        \item~$\rho^\T x + \rho_0  \theta \geq \sigma_{\feasreg}(\rho, \rho_0)$ is implied by the inequalities in items (i) and (ii). \label{item:langrangian3}
    \end{enumerate}
    \label{thm:lagrangian_cut}
\end{theorem}
\proof
    By summing the inequalities in~\ref{item:langrangian1} and~\ref{item:langrangian2}, it suffices to show that~$\sigma_{\feasreg}(\rho, \rho_0) = -\hat{\alpha}^\T h + \sigma_{X}(\rho + T^ \T \hat{\alpha}) + \sigma_{\epifY} (\hat{\alpha}, \rho_0)$. Since~$X$ and~$Y$ are nonempty polyhedra, it follows from Lagrangian duality~\citep{boyd2004convex} that
    \begin{align*}
        \sigma_{\feasreg}(\rho, \rho_0) = &~\min_{x, y} && \rho^\T x + \rho_0  (d^ \T y) \\
        & && h - Tx = Qy, \\
        &&& (x, y) \in X \times Y, \\
        = & \max_{\alpha \in \R^p} ~ \min_{x, y}~~ && \rho^\T x + \rho_0  (d^ \T y) + \alpha^\T (Qy - (h - Tx)) \\
        &&& (x, y) \in X \times Y, \\
        = &~\max_{\alpha \in \R^p} \left\{ - \alpha^ \T h + \sigma_{X}(\rho + T^ \T \alpha) + \sigma_{Y}(Q^\T \alpha + \rho_0  d)\right\} \span\span \\
        = &~ -\hat{\alpha}^\T h + \sigma_{X}(\rho + T^ \T \hat{\alpha}) + \sigma_{Y}(Q^\T \hat{\alpha} + \rho_0  d) \span \span & \text{(by optimality of~$\hat{\alpha}$)} \\
        = &~ - \hat{\alpha}^\T h + \sigma_{X}(\rho + T^ \T \hat{\alpha}) + \sigma_{\epifY} (\hat{\alpha},\rho_0). \span \span & \text{(by Lemma~\ref{lemma:support})}
        \tag*{\Halmos}
    \end{align*}

Based on Theorem~\ref{thm:lagrangian_cut}, we henceforth say that a \textit{Lagrangian cut} is an inequality of the form~$\hat{\alpha}^\T (h - Tx) + \theta \geq \sigma_{\epifY} (\hat{\alpha}, 1)$ where~$\hat{\alpha} \in \R^p$ is optimal for the Lagrangian dual problem~\eqref{problem:lagrangian2}, with~$\rho = c$ and~$\rho_0 = 1$. Therefore, a Lagrangian cut implies the objective function cut~$c^\T x + \theta \geq \sigma_{\feasreg}(c,1) = z^*$. Proving Proposition~\ref{prop:strong_corner} is now immediate.

\begin{myproof}{of Proposition~\ref{prop:strong_corner}}
    Since~$\epifY[C]$ is a relaxation of~$\epifY$, the optimal value of the problem in the left-hand side of~\eqref{problem:corner_bound} is at most~$z^*$. For the converse, apply Theorem~\ref{thm:benders_corner} with the Lagrangian cut~$\hat{\alpha}^\T (h - Tx) + \theta \geq \sigma_{\epifY} (\hat{\alpha}, 1)$.
\end{myproof}

\section{\mytitle{Network flow structure and the Dantzig-Wolfe bound}}
\label{section:vrpsd}

In this section, we focus on the special case where the inequalities defining~$Y$ arise from a network flow formulation. We start with some basic definitions. As usual~\citep{cook2011combinatorial, ahuja1988network}, 
given a finite set $\mathcal{S}$, we index vectors by the elements of the set:
for a set $\mathbb{D}$ (e.g., $\R$ or $\Z$),
we denote by $\mathbb{D}^\mathcal{S}$ a vector indexed by $\mathcal{S}$ where each entry has domain $\mathbb{D}$.

\begin{definition}
    A \textit{network} is a tuple \( (\mathcal{N}, b, u) \), where $ \mathcal{N} $ is a directed graph identified by a tuple $(\mathcal{V}, \mathcal{A})$ of its nodes $\mathcal{V}$ and arcs $\mathcal{A}$; \( b \in \Z^{\mathcal{V}} \) specifies the demand/supply values at each node~$v \in \mathcal{V}$; and~$u \in \Z^{\mathcal{A}}$ is a vector of arc capacities. Let $\nodearcmatrix \in \R^{{\mathcal{V}} \times \mathcal{A}}$ denote the node-arc incidence matrix of~$\mathcal{N}$, that is, for each~$v \in \mathcal{V}$ and~$a \in \mathcal{A}$, the entry~$\nodearcmatrix_{v,a}$ has value~$+ 1$ if~$v$ is the head of~$a$;~$-1$ if~$v$ is the tail of~$a$; and~$0$ otherwise.
    \label{definition:network}
\end{definition}

\begin{definition}
    Let~$(\mathcal{N} = (\mathcal{V}, \mathcal{A}), b, u)$ be a network.
    Then $\mathcal{Y} \subseteq \R^m$ is a \emph{network flow polytope associated with~$(\mathcal{N}, b, u)$} if it holds that~$\mathcal{Y} = \left\{ y \in\R^{\mathcal{A}}_+ : \nodearcmatrix y = b, y \leq u \right\}$.
    \label{def:network_flow_polytope}
\end{definition}
\noindent
Without loss of generality, we will assume that every network~$\mathcal{N}$ mentioned henceforth is connected.

In the remainder of this section, we consider value functions $\fY[\mathcal{Y}]$, where $\mathcal{Y}$ is a network flow polytope associated with a network~$(\mathcal{N}, b, u)$. Our goal here is threefold: 
    in Section~\ref{subsection:path_flow}, we apply the flow-decomposition theorem to establish a connection between such value functions and DW reformulations; 
    in Section~\ref{subsec:recover-dw-bound}, we observe that corner Benders' cuts can recover the bound from the DW formulation; and
    in Sections~\ref{subsection:vrpsd} and~\ref{subsection:corner_example}, we formulate a DW-based linear programming relaxation for the Vehicle Routing Problem with Stochastic Demands (VRPSD) in the form of Problem~\eqref{problem_master}, 
    and we offer a simple VRPSD example to illustrate the generation of corner Benders' cuts.

\subsection{\mytitle{Equivalence between path-flow and arc-flow formulations}}
\label{subsection:path_flow}

Let~$\Pi$ be a finite set of vectors, which we interpret as the incidence vectors of feasible ``patterns'' for a combinatorial optimization problem, e.g., a route for a vehicle routing problem. For simplicity, we assume that~$\Pi$ is a subset of~$\R^n$, but one can check that the following reasoning also holds for the general case of~$\Pi \subseteq \R^p$ (as long as we later choose matrices~$Q$ and~$T$ appropriately). Next, we follow the presentation in~\cite{de2022arc} to show that network flows can be used to optimize over the set~$X \cap (k \cdot \conv(\Pi))$, where~$k \in \Z_{++}$ denotes the number of patterns to be selected and~$k \cdot \conv(\Pi) \defeq \{k \pi : \pi \in \conv(\Pi)\}$. For each vector~$\pi \in \Pi$, let~$\hat{d}_\pi$ be a rational number representing the ``cost'' of using the pattern associated with~$\pi$. 
We can minimize a linear function over $x \in X \cap (k \cdot \conv(\Pi))$ as follows.
\begin{equation}
\begin{aligned}
    \hat{z} \defeq \min_{x, \lambda} ~~& c^{\T} x + \sum_{\pi \in \Pi} \hat{d}_\pi \cdot \lambda_\pi & %
    \\
    & x = \sum_{\pi \in \Pi} \pi \cdot \lambda_\pi, & %
    \\
    & \sum_{\pi \in \Pi} \lambda_\pi = k, %
    \\
    & x \in X, & %
    \\
    & \lambda_\pi \geq 0, & \forall \pi \in \Pi. & 
\end{aligned}
\label{problem_dw}
\end{equation}

\noindent
Typically, Problem~\eqref{problem_dw} is solved in a \emph{column generation} framework,
where the variables $\lambda_\pi$ are added iteratively via a \emph{pricing subproblem}.

One can always construct an acyclic directed graph~$\mathcal{N} = (\mathcal{V}, \mathcal{A})$, with source vertex~$s \in \mathcal{V}$ and sink vertex~$t \in \mathcal{V} \setminus \{s\}$, such that there exists a one-to-one mapping between the~$s-t$ paths in~$\mathcal{N}$ and the elements of~$\Pi$. Furthermore, we can set~$d \in \R^{\mathcal{A}}$ and~$Q \in \R^{n \times \mathcal{A}}$ so that, for each~$\pi \in \Pi$, with corresponding~$s-t$ path~$P$, we have that~$\sum_{a \in P} d_a = \hat{d}_\pi$ and~$\sum_{a \in P} Q_a = \pi$ (the summation is over the arcs in~$P$ and~$Q_a \in \R^n$ is the column of~$Q$ indexed by~$a \in \mathcal{A}$). This can be accomplished trivially by setting~$\mathcal{A}$ to be made of~$|\Pi|$ parallel arcs between~$s$ and~$t$. More interestingly, whenever optimizing a linear function over~$\Pi$ can be done with a \emph{dynamic programming} (DP) algorithm~$\mb{A}$, we can set the vertices (respectively, arcs) of~$\mathcal{N}$ to correspond to the DP states (respectively, DP state transitions) of algorithm~$\mb{A}$. In this way, the~$s-t$ paths in~$\mathcal{N}$ correspond to feasible outputs of~$\mb{A}$, and therefore, to elements of~$\Pi$ (see~\cite{de2022arc} for more details and examples).

The previous argument shows that we can interpret Problem~\eqref{problem_dw} as a \textit{path-flow} formulation, where each variable~$\lambda_\pi$ corresponds to a flow variable over an~$s-t$ path in~$\mathcal{N}$. Since~$\mathcal{N}$ is acyclic, we can apply the flow decomposition theorem (see Section 3.5 of~\cite{ahuja1988network}) to map each path-flow solution~$\lambda \in \R^{\Pi}$ to a corresponding \textit{arc-flow} solution~$y \in \R^{\mathcal{A}}$ and vice versa. 
For any~$v \in \mathcal{V}$, let~$\allones_v \in \R^{\mathcal{V}}$ be the characteristic vector of~$\{v\}$. %
It follows that, with 
$Y$ representing the network flow polytope associated with~$\mathcal{N}$,
the optimal value $\hat{z}$ of Problem~\eqref{problem_dw} can be equivalently obtained by
\begin{equation}
    \min_{x,y} \left\{ 
        c^{\T} x + d^\T y 
        \suchthat
        x = Qy,\,
        x \in X,\,
        y \in Y = \{y \in \R^{\mathcal{A}}_+ \suchthat \nodearcmatrix y = k \cdot \allones_t - k \cdot \allones_s\}
    \right\}.
\label{problem_flow}
\end{equation}
Hence, with $h = 0$ and $T$ as the negative identity matrix, via Problem~\eqref{problem_master}, we can attain the same bound $\hat{z}$, even if we are only allowed to use network flow polytopes for $Y$.
As a result, in what follows, we return to the $x$ variable space because $x = w$ in our setting.

We also mention in passing that this path-flow/arc-flow equivalence also relates our work with important studies on cut generation via {decision diagrams}~\citep{becker2005, behle_thesis, tjandraatmadja2019target, castro2022combinatorial, lozano2022binary}, but in Section~\ref{subsec:recover-dw-bound} we position our contributions in the context of DW formulations.

\subsection{\mytitle{Recovering the DW bound}}
\label{subsec:recover-dw-bound}

Consider the setting where the set~$\Pi$ is given by the set of extreme points of a polytope. In this case, Problem~\eqref{problem_dw} encompasses different modeling paradigms for DW formulations, and the value~$\hat{z}$ corresponds to the so-called \emph{DW bound}. For example, when~$\hat{d} \equiv 0$, then the problem falls in a format similar to the \textit{explicit master problem} in~\cite{de2003integer} and also the general DW decomposition formulation (such as in~\cite{chen2024recovering}). In addition, one may use $\hat{d}_\pi$ to model a ``second-stage'' cost associated with the~$\lambda$ variables; this situation occurs, for instance, in column generation formulations for two-stage stochastic programming problems~\citep{christiansen2007, GAUVIN2014141, silva2006solving}.

The case of DW formulations is relevant because it is well known that, for many combinatorial optimization problems, DW bounds can be significantly stronger than the bounds arising from LP relaxations of natural formulations in lower-dimensional spaces~\citep{pessoa2020generic, de2022arc}. However, algorithms based on DW formulations typically use a branch-and-price scheme, which is not natively supported by many commercial MIP solvers. Furthermore, as highlighted in~\cite{uchoa2024optimizing}, most MIP solvers have a native support for the addition of cuts, but not for the addition of the variables needed in a branch-and-price method. In this context, \cite{chen2024recovering} recently proposed an approach that recovers the DW bound with cuts in the original space of the variables. We argue next that our approach extends the results of~\cite{chen2024recovering} by providing multiple corner Benders' cuts that recover the DW bound. Our computational experiments show that this is indeed advantageous, allowing us to obtain state-of-the-art results for
a well-studied combinatorial optimization problem.

As discussed earlier, let $\mb{A}$ be a DP algorithm that solves the pricing subproblem with respect to Problem~\eqref{problem_dw} and~$\mathcal{N}$ be a corresponding network.
Setting~$Y$ as in Problem~\eqref{problem_flow}, and recalling that $x=w$ in our setting, yields
    $\hat{z} = \min_{x,\theta} \left\{c^{\T} x + \theta \suchthat x \in X, (x, \theta) \in \epifY \right\}$. 
Solving Problem~\eqref{problem:lagrangian1}, we obtain a Lagrangian cut $\alpha^\T x + \theta \geq \sigma_{\epifY}(\alpha, 1)$ , and by Theorem~\ref{thm:lagrangian_cut} with $(\rho,\rho_0) = (c, 1)$, this inequality implies the objective function cut $c^\T x + \theta \ge \sigma_{\feasreg}(c,1) = \hat{z}$.
This is essentially the same reasoning as the one devised by~\cite{chen2024recovering} to recover the DW bound via cuts (see also Appendix~\ref{appendix:block_diagonal} for the case of multiple ``blocks''). Proposition~\ref{prop:strong_corner} then shows that instead of recovering the DW bound with a single Lagrangian cut, we can recover this bound with multiple corner Benders' cuts. As we argued earlier in Section~\ref{subsection:corner_bounds}, this can be advantageous because it might reduce the dimension of an optimal face.

\subsection{\mytitle{Application: the vehicle routing problem with stochastic demands (VRPSD)}}
\label{subsection:vrpsd}

In this section, we introduce the combinatorial optimization problem that will be used in our computational experiments, namely the Vehicle Routing Problem with Stochastic Demands (VRPSD). The reason we introduce the problem at this point is so that we can give concrete examples
of some of the concepts that
have been discussed, which will hopefully help the reader to better understand them.

The VRPSD is a variant of the classic capacitated vehicle routing problem where, rather than deterministic, customer demands are stochastic and follow a given probability distribution. This problem has a very rich literature, with many studies in the last few years~\citep{gendreau201650th, florio2020, florio2022recent, hoogendoorn2023improved, parada2024disaggregated, ota2024hardness}. As usual, we model the VRPSD as a \textit{two-stage stochastic program}, where a recourse policy describes the recourse actions that
should be taken in the case of \textit{route failures}, i.e., when a violation of the vehicle capacity constraint is observed. The
objective function then combines the initial planned routes cost (\textit{first-stage cost}) and the
expected recourse cost (\textit{second-stage cost}). Next, we present a formal definition.

\subsubsection{\mytitle{Definition of the VRPSD}}
\label{subsection:vrpsd_definition}

Let~$G = (V = \{0\} \cup V_+, E)$ be a complete undirected graph. The vertex~$0$ represents the \textit{depot} and the set~$V_+$ indicates the \textit{customers}. The vector~$c \in \Q^E_{++}$ denotes the edge costs, $k \in \Z_{++}$ refers to the desired number of vehicles, and $C \in \Q_{++}$ is the vehicle capacity. Let~$q$ be a random vector over~$V_+$, where each entry~$q_v$ is a random variable corresponding to the demand of customer~$v \in V_+$. The vector~$q$ is governed by a probability distribution~$\mb{P}$, and we define~$\bar{q} \defeq \mb{E}[q]$. A \textit{route}~$R$ is a tuple of customers~$(v_1, \ldots, v_\ell)$ and~$R$ is \textit{elementary} if the customers in~$R$ are all distinct. We use~$q(R)$ to refer to the sum of random variables~$\sum_{i = 1}^\ell q_{v_i}$, and we say that route~$R$ is a \textit{q-route} if the expected total demand of the route is below capacity, i.e., if $\bar{q}(R) \defeq \mb{E}[q(R)] \leq C$. The notation~$c(R)$ denotes the \textit{first-stage cost of route~$R$}, i.e.,~$c(R) \defeq c_{0 v_1} + c_{0 v_\ell} + \sum_{j = 1}^{\ell - 1} c_{v_j v_{j + 1}}$ ($c(R) = 2 c_{0 v_1}$ if~$\ell = 1$). Moreover, with each route~$R$ we associate a random variable~$\mathcal{Q}(R)$ and the expectation~$\mb{E}[\mathcal{Q}(R)]$ denotes the (expected) \textit{second-stage cost of~$R$}.

The definition of the random variable~$\mathcal{Q}(R)$ depends on the choice of the recourse policy. In this work, we focus on the \textit{classical recourse policy}, that is, whenever a failure is observed, the recourse policy determines that the vehicle should execute a back-and-forth trip between the depot and the failed customer. In this context, \cite{dror89} show that, for any q-route~$R = (v_1, \ldots, v_\ell)$, the second-stage cost can be calculated with the following formula,
which computes the probability that $t$ returns to the depot are needed to satisfy customer $j$'s demand.
We use the shorthand $[i]$ for an integer $i \ge 0$ to denote the set $\{1,\ldots,i\}$ and define $[0] \defeq \emptyset$.
\begin{equation}
    \mb{E}[\mathcal{Q}(R)] 
        = 
        2 \sum_{j \in [\ell]} \sum_{t = 1}^{\infty} 
            \mb{P}\left[
                \sum_{i \in [j - 1]} q_{v_i} \leq t C < \sum_{i \in [j]} q_{v_i}
            \right] c_{0v_j}.
    \label{eq:expected_cost}
\end{equation}
The goal of the VRPSD is to find a set of~$k$ distinct elementary q-routes~$\{R_1, \ldots, R_k\}$ that forms a partition of the customers~$V_+$ and minimizes the sum~$\sum_{i \in [k]} (c(R_i) + \mb{E}[\mathcal{Q}(R_i)])$.

\newcommand{\subfigAcaption}{
  A set of q-routes~$\{R_1, R_2\}$ that minimizes the first-stage cost. Here we have that~$\sum_{i \in [2]} c(R_i) = 76$.
}
\newcommand{\subfigBcaption}{
  A solution~$\{R'_1, R'_2\}$ that is optimal for the VRPSD. The sum of the first and second-stage costs is~$\sum_{i \in [2]} (c(R'_i) + \mb{E}[\mathcal{Q}(R'_i)]) = 88$.
}
\begin{figure}[htb]
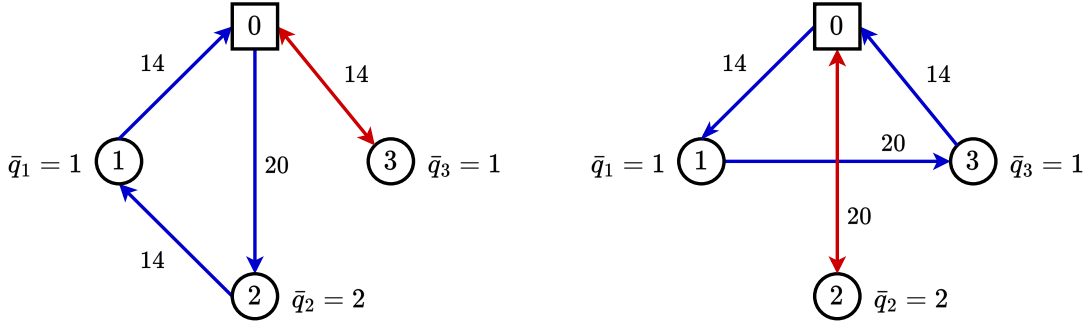

    \centering
    \subfloat[\subfigAcaption]{
        \includesvg[width=0.46\linewidth]{images/graph1_corner}
        \label{fig:vrpsd-a}
    }
    \hfill
    \subfloat[\subfigBcaption]{
        \includesvg[width=0.46\linewidth]{images/graph2_corner}
        \label{fig:vrpsd-b}
    }
    \\[0.5cm]
    \caption{
        Example of a VRPSD instance with~$V_+ = \{1, 2, 3\}$, $k = 2$, and $C = 3$. 
        The edge costs are $c_{01} = c_{03} = c_{12} = c_{23} = 14$ and $c_{13} = c_{02} = 20$. 
        The expected demand values $\bar{q}_v$ are given beside each customer $v \in V_+$. %
        In this example, each $q_i$ follows a normal probability distribution with variance~$0.1$, 
        so if a q-route~$R$ has $\bar{q}(R) \leq 2$, 
        then~$\mb{E}[\mathcal{Q}(R)] \approx 0$. 
        The solution in (a) has total cost~$\sum_{i \in [2]} (c(R_i) + \mb{E}[\mathcal{Q}(R_i)]) = 76 + \mb{E}[\mathcal{Q}((2, 1))] = 76 + (1/2) \cdot 28 = 90$, where the 1/2 appears because the normal distribution is symmetric and results in 50\% chance of exceeding capacity at customer 1.
    }
    \label{figure:vrpsd}
\end{figure}

Throughout this paper, we only consider instances of the VRPSD that satisfy Assumption~\eqref{assumption:vrpsd}, which says that, at any point along a route, the second-stage cost is just a function of the last vertex visited and the accumulated expected demand. %

\begin{assumption*}[Assumption~($\star$).]
    There exists a function~$\Psi : \R \times V_+ \to \R$ such that, for any elementary q-route~$R = (v_1, \ldots, v_\ell)$,
        \begin{equation}
            \tag{$\star$}\label{assumption:vrpsd}
            \Psi(v_\ell, \bar{q}(R)) 
                = 2c_{0v_\ell} \sum_{t = 1}^\infty \mb{P}\left[\sum_{i \in [\ell - 1]} q_{v_i} \leq t C < \sum_{i \in [\ell]} q_{v_i}\right] .
        \end{equation}
\end{assumption*}
    
    In particular, this implies that for any elementary q-route~$R = (v_1, \ldots, v_\ell)$,
        \begin{equation*}
            \mb{E}[\mathcal{Q}(R)] = \sum_{j \in [\ell]} \Psi(v_j, \bar{q}((v_1, \ldots, v_j))).
        \end{equation*}

Such an assumption is common in the VRPSD literature~\citep{christiansen2007, GAUVIN2014141, florio2020}, as it facilitates a column generation formulation of the VRPSD where the pricing subproblem is solved via a ``knapsack-like'' DP algorithm that defines one state for every possible choice of $(v_\ell, \bar{q}(R))$
(observe that multiple routes might collapse to the same state in this construction). 
Moreover, branch-and-cut algorithms~\citep{jabali2014, hoogendoorn2023improved,  parada2024disaggregated} consider instances that either satisfy Assumption~\eqref{assumption:vrpsd}, or a similar assumption where function~$\Psi$ also depends on other parameters such as the total variance of route~$R$.

\subsubsection{\mytitle{Formulation via network flows}}
\label{subsection:vrpsd_formulation}

We now present a mixed-integer programming formulation for the VRPSD whose LP relaxation has the form of Problem~\eqref{problem_master}. We begin by building on the branch-and-price algorithm of~\cite{christiansen2007} and the connection mentioned in Section~\ref{subsection:path_flow}. 

We construct a set partitioning formulation of the VRPSD via Problem~\eqref{problem_dw}.
Let~$\mathcal{R}$ be the set of all q-routes with respect to the input. 
Additionally, for each route~$R \in \mathcal{R}$ and edge~$e \in E$, let $\tsc{count}(R, e)$ denote how many times edge~$e$ appears in route~$R$. 
For each route $R \in \mathcal{R}$, let $\pi(R) \in \Z^E$ be a vector where each entry has value $\tsc{count}(R,e)$,
and let $\Pi = \{\pi(R) \suchthat R \in \mathcal{R}\}$.
We define $\hat{d}_{\pi(R)} = \mb{E}[\mathcal{Q}(R)]$ for each $R \in \mathcal{R}$,
$k$ as the number of vehicles in the VRPSD instance,
and
\begin{equation*}
    X =
    \left\{x \in [0, 2]^{E} \suchthat
    \begin{aligned}
        & x(\delta(0)) = 2 k, \\
        & x(\delta(v)) = 2, & \forall v \in V_+ \\
        & x(\delta(S)) \geq 2 \left\lceil \frac{\bar{q}(S)}{C} \right\rceil, & \forall S \subseteq V_+, S \neq \emptyset \\
        & x_e \leq 1, & \forall e \in E \setminus \delta(0)
    \end{aligned}
    \right\}.
\end{equation*}
Here, $\delta(v)$ denotes the set of edges incident to $v$,
and we allow edges to be traversed twice (for a route that visits a single customer).

Note that the linking constraints (which take the form $x_e = \sum_{R \in \mathcal{R}} \tsc{count}(R, e) \cdot \lambda_R$, for all $e \in E$)
and the definition of~$X$ imply that~$\lambda_R \leq 1$, for all~$R \in \mathcal{R}$. 
We refrain from further explaining this set partitioning formulation, as it follows the format of relatively standard formulations used in various branch-cut-and-price algorithms for different vehicle routing variants, including the VRPSD~\citep{Fukasawa2006, pessoa2020generic, christiansen2007, GAUVIN2014141, florio2020}.

\begin{figure}[htb]
    \centering
    \includesvg[width=0.7\linewidth]{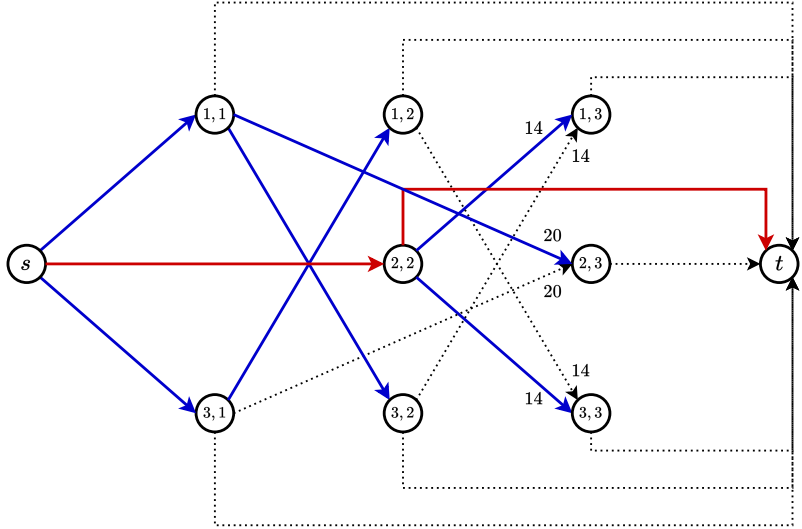}
    \caption{Network~$\mathcal{N} = (\mathcal{V}, \mathcal{A})$ corresponding to the example shown in Figure~\ref{figure:vrpsd}. The numbers next to the arcs indicate the arc cost~$d_a$. We omit these numbers when the arc cost is zero. 
    From left to right, each layer $\ell$ (indexed by the second number of each node label, corresponding to accumulated demand) contains a node for every customer with $\bar{q}_v \le \ell$.}
    \label{figure:vrpsd_network}
\end{figure}

Given Assumption~\eqref{assumption:vrpsd}, %
iteratively finding routes (pricing in a column generation framework) for the problem
can be modeled as a shortest~$s-t$ path problem in a network~$\mathcal{N} = (\mathcal{V}, \mathcal{A})$ as follows (see Figure~\ref{figure:vrpsd_network}). 
The vertex set~$\mathcal{V}$ is associated to ``states'' and comprises a source vertex~$s$ (state $[0,0]$), a sink vertex~$t$, and there is a vertex/state $\mathcal{S} = [v, \mu]$ for each $v \in V_+$ and $\mu \in \{\bar{q}_v, \ldots, C\}$. 
Intuitively, state~$[v, \mu]$ corresponds to a DP state of the pricing algorithm where we are currently at customer~$v$ with an accumulated average demand of~$\mu$. 
With this interpretation in mind, we define~$\vertex : \mathcal{V} \to V$ as a mapping such that~$\vertex(\mathcal{S}) = v$, if~$\mathcal{S} = [v, \mu]$, and~$\vertex(\mathcal{S}) = 0$, if~$\mathcal{S} \in \{s, t\}$.

Each arc~$\mathcal{S}_1 \mathcal{S}_2 \in \mathcal{A}$ is mapped to an edge~$\edge(\mathcal{S}_1 \mathcal{S}_2) \defeq \{\vertex(\mathcal{S}_1), \vertex(\mathcal{S}_2)\} \in E$. The intuition is that taking arc~$\mathcal{S}_1 \mathcal{S}_2$ in an~$s-t$ path in~$\mathcal{N}$ corresponds to adding~$\edge(\mathcal{S}_1 \mathcal{S}_2)$ to a route in the VRPSD instance graph $G$.
The arc set~$\mathcal{A}$ of the network is given by the union of the following sets
\begin{flalign*}
\SRC & = \{(s, [v, \bar{q}_v]) \suchthat v \in V_+\}, &\\
\FWD & = \{([u, \mu], [v, \mu + \bar{q}_v]) \suchthat u, v \in V_+, \mu \in \{\bar{q}_u, \ldots, C - \bar{q}_v\} \},~\text{and} &\\
\SNK & = \{([v, \mu], t) \suchthat v \in V_+, \mu \in \{\bar{q}_v, \ldots, C\}\}. &
\end{flalign*}
In practice, any state~$\mathcal{S}$ that is not reachable from the source~$s$ will be removed in a preprocessing step.

We set the arc costs~$d \in \R^{\mathcal{A}}$ as
\begin{equation*}
    d_a =
    \begin{cases}
        \Psi(v, \mu + \bar{q}_v), & \text{if }a = ([u, \mu], [v, \mu + \bar{q}_v]) \in \SRC \cup \FWD \\
        0, & \text{otherwise.}
    \end{cases}
\end{equation*}

Let~$\mb{I}(\,\cdot\,)$ denote the indicator function.
Applying the derivation of Problem~\eqref{problem_flow} in Section~\ref{subsection:path_flow}, we obtain the following flow-based formulation that is equivalent to the earlier set partitioning formulation based on Problem~\eqref{problem_dw},
where
\begin{equation*}
    Y = \left\{
        y \in \R_+^{\mathcal{A}} \suchthat
        y(\delta^-(\mathcal{S})) - y(\delta^+(\mathcal{S})) =
            k \cdot \mb{I}(\mathcal{S} = t) - k \cdot \mb{I}(\mathcal{S} = s),
            ~\forall \mathcal{S} \in \mathcal{V}
    \right\}
\end{equation*}
\begin{equation}
    z_{SP} = 
        \min_{x,y} \left\{
            \sum_{e \in E} c_e x_e + \sum_{a \in \mathcal{A}} d_a y_a
            \suchthat
            \sum_{a \in \mathcal{A} : \edge(a) = e} y_a = x_e \text{ for all } e \in E,\,
            x \in X,\,
            y \in Y
        \right\}.
    \label{problem:flow_cl}
\end{equation}

Let~$Q \in \R^{E \times \mathcal{A}}$ be such that each column~$Q_a$, with~$a \in \mathcal{A}$, is the incidence vector of~$\edge(a) \in E$. Then
    \begin{equation*}
        z_{SP} 
            = \min_{x,y} \{
            c^\T x + d^\T y \suchthat x = Qy, x \in X, y \in Y\} 
            = \min_{x,\theta} \{c^\T x + \theta \suchthat x \in X, (x, \theta) \in \epifY\},
    \end{equation*}
and we can solve the VRPSD to integrality by enforcing that~$x \in \Z^E$ above.

\subsection{\mytitle{Corners of network flow polytopes and numerical example of corner Benders' cuts}}
\label{subsection:corner_example}

In this subsection, we shall see that Figure~\ref{figure:vrpsd_network} can be interpreted as an illustration of a corner~$C$ of the network flow polytope~$Y$ defined in Section~\ref{subsection:vrpsd_formulation}. To do this, we first recall some basic facts from network flow theory (see Section 11.11 of~\cite{ahuja1988network}) %
to establish a connection between the geometric notion of corners and the combinatorial concepts of spanning trees and cycles.

\subsubsection{\mytitle{Network simplex background}}
One can show that exactly one of the equations in the system~$\nodearcmatrix y = b$ is redundant. Thus, from now on, we use~$\nodearcmatrix' y = b'$ to refer to the system of equations obtained by removing the equation associated with vertex~$s \in \mathcal{V}$. 
For any subset of arcs~$\mathcal{B}$ of~$\mathcal{A}$ we write~$\nodearcmatrix[N_{\mathcal{B}}]$ (respectively $\nodearcmatrix[N_{\mathcal{B}}']$) to denote the submatrix of $\nodearcmatrix$ (respectively $\nodearcmatrix'$) obtained by dropping the columns associated with arcs in $\mathcal{A} \setminus \mathcal{B}$. The following results are well known in network flow and network simplex theory.
\begin{theorem}
    Let~$\mathcal{T} \subseteq \mathcal{A}$ be such that~$\cardinality{\mathcal{T}} = \cardinality{\mathcal{V}} - 1$. The columns of~$\nodearcmatrix[N_{\mathcal{T}}']$ are linearly independent if and only if the arcs in~$\mathcal{T}$ forms a spanning tree of~$\mathcal{N}$.
    \label{thm:flow-tree}
\end{theorem}
\begin{definition}
    Let~$C^a \subseteq \mathcal{N}$ be a cycle containing arc~$a \in \mathcal{A}$. 
    The notation~$\allones(C^a) \in \R^{\mathcal{A}}$ denotes the incidence vector of~$C^a$, i.e., for each $a' \in \mathcal{A}$,
    \begin{equation*}
        (\allones(C^a))_{a'} \defeq
        \begin{cases}
            1, & \text{if~$a'$ is a forward arc in~$C^a$}, \\
            -1, & \text{if~$a'$ is a backward arc in~$C^a$}, \\
            0, & \text{otherwise.} \\
        \end{cases}
    \end{equation*}
    \label{definition:cycle-ray}
\end{definition}
\begin{fact}
    Let~$\mathcal{T}$ be a basis (spanning tree) of~$Y$. For each~$a \in \mathcal{A} \setminus \mathcal{T}$,
    define a vector~$r^{a, \mathcal{T}} \in \R^{\mathcal{A}}$ as follows
    $$
    r^{a, \mathcal{T}} \defeq
    \begin{pmatrix}
    (\nodearcmatrix[N_{\mathcal{T}}'])^{-1} \nodearcmatrix'_{\{a\}} \\
    1 \\
    0 \\
    \end{pmatrix}
    \begin{array}{llll}
    \} ~\mathcal{T} \\
    \} ~\{a\} \\
    \} ~(\mathcal{A} \setminus \{a\}) \setminus \mathcal{T},
    \end{array}
    $$
    where the sets in the right indicate the corresponding entries of the vector.
    Then~$r^{a, \mathcal{T}} = \allones(C^a)$, where~$C^a$ is the unique cycle in the graph induced by~$\mathcal{T} \cup \{a\}$.
    \label{fact:flow-ray}
\end{fact}

\noindent
Henceforth, whenever~$\mathcal{T}$ is clear from the context, we omit~$\mathcal{T}$ from the superscript of~$r$.

\subsubsection{\mytitle{Numerical example}}

We now present a numerical example that illustrates corner Benders' cuts. Consider the instance and network illustrated in Figures~\ref{figure:vrpsd} and~\ref{figure:vrpsd_network}. 
Suppose that our current solution~$(\bar{x}, \bar{\theta})$ has~$\bar{x}_{01} = \bar{x}_{02} = \bar{x}_{12} = 1$,~$\bar{x}_{03} = 2$ and~$\bar{\theta} = 0$ (same as Figure~\ref{fig:vrpsd-a}, but with~$\bar{\theta}$ less than~$\fY(\bar{x}) = \mb{E}[\mathcal{Q}((2,1))] = (1/2) \cdot 28 = 14$).
We first solve Problem~\eqref{problem:lagrangian1} to obtain the Lagrangian cut
\begin{equation}
    (\alpha^1)^\T x + \theta = x_{01} + x_{03} -15 x_{12} - 2 x_{13} - 15 x_{23} + \theta \geq 0.
    \label{ineq:vrpsd_lagrange}
\end{equation}
Note that~$(\bar{x}, \bar{\theta})$ violates this inequality: $(\alpha^1)^\T \bar{x} + \bar{\theta} = -12 < 0$.
Then, via Theorem~\ref{thm:lagrangian_cut} with $(\rho,\rho_0) = (c,1)$, \eqref{ineq:vrpsd_lagrange} implies the objective function cut
    $c^\T x + \theta \geq \sigma_{\feasreg}(c,1) = z_{SP}$, and in this case, it turns out that $z_{SP} = 88$.
We then solve
    $\min_{y} \{(\alpha^1)^\T Q y + d^\T y \suchthat y \in Y \}$ 
to obtain an optimal solution~$y^*$ and a corresponding optimal basis (spanning tree)~$\mathcal{T}$. These are illustrated in Figure~\ref{figure:vrpsd_network}: the solid arcs indicate the arcs in the tree and~$y^*$ is given by the red arcs (i.e., path~$(s, [2, 2], t)$), so~$\sigma_Y(Q^\T \alpha^1 + d) = 0$.

Let~$C = \{y^*\} + \cone(R)$ be a corner of~$Y$ derived from the basis~$\mathcal{T}$ according to Remark~\ref{remark:simplex}. By Fact~\ref{fact:flow-ray}, we have~$R = \{r^a \}_{a \in \mathcal{A} \setminus \mathcal{T}}$, that is, the rays in Remark~\ref{remark:simplex} correspond to the cycles that we obtain by adding arcs in~$\mathcal{A} \setminus \mathcal{T}$ to the spanning tree~$\mathcal{T}$. Moreover, using Lemma~\ref{lemma:epigraph_cone}, we have that~$\epifY[C] = \{(Qy^*, d^\T y^*)\} + \cone(\{(Q r^a , d^\T r^a \}_{a \in \mathcal{A} \setminus \mathcal{T}} \cup \{(\mathbf{0}, 1)\})$.

\begin{figure}[htb]
    \centering
    \includesvg[width=0.7\linewidth]{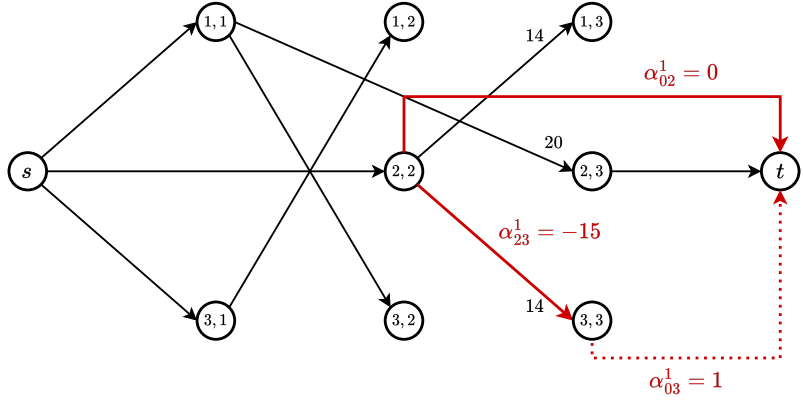}
    \caption{Illustration of the ray (cycle)~$r^{([3, 3], t)}$. Notice that~$(\alpha^1)^\T (Q r^{([3, 3], t)}) + d^\T r^{([3, 3], t)}) = 0$.}
    \label{figure:vrpsd_cycle}
\end{figure}

In this way, we apply Theorem~\ref{thm:polar} to separate the corner Benders' cut
\begin{equation}
    (\alpha^2)^\T x + \theta = -14 x_{12} - 14 x_{23} + \theta \geq 0.
    \label{ineq:vrpsd_corner}
\end{equation}

\noindent
Observe that~\eqref{ineq:vrpsd_corner} dominates~\eqref{ineq:vrpsd_lagrange} since summing some of the degree constraints defining~$X$ to~\eqref{ineq:vrpsd_lagrange} yields
\begin{align}
& (\alpha^1)^\T x + \theta - 2 x(\delta(0)) + x(\delta(1)) + x(\delta(3)) \geq -2 (2k) + 2 (2) \nonumber \\
\iff & (\alpha^1)^\T x + \theta + (-2 x_{01} - 2 x_{02} - 2 x_{03}) + (x_{01} + x_{12} + x_{13}) + (x_{03} + x_{13} + x_{23}) \geq -4 \nonumber \\
\iff & -2 x_{02} - 14 x_{12} - 14 x_{23} + \theta \geq -4, \nonumber
\end{align}
which is equivalent to summing~\eqref{ineq:vrpsd_corner} with inequality~$- 2 x_{02} \geq -4$. In addition, we can manually verify that, as predicted by Theorem~\ref{thm:polar}, inequality~\eqref{ineq:vrpsd_corner} defines a facet of~$\epifY[C]$. To see validity, we can just check that~$(\alpha^2)^\T (Q y^*) + d^\T y^* = 0$ and~$(\alpha^2)^\T (Q r^a) + d^\T r^a \geq 0$, for all~$a \in \mathcal{A} \setminus \mathcal{T}$. To check the dimension of the face induced by~\eqref{ineq:vrpsd_corner} we observe that~$\alpha = \alpha^2$ is the unique solution to the system of equations below.
\begin{align}
& \alpha^\T (Q y^*) + d^\T y^* = 2 \alpha_{02} + 2 \alpha_{02} = 0 \nonumber \\
& \alpha^\T Q r^{([1, 1], t)} + d^\T r^{([1, 1], t)} = 2 \alpha_{01} - 2 \alpha_{02} = 0, \nonumber \\
& \alpha^\T Q r^{([1, 2], t)} + d^\T r^{([1, 2], t)} = \alpha_{01} - 2 \alpha_{02} + \alpha_{03} + \alpha_{13} = 0, \nonumber \\
& \alpha^\T Q r^{([1, 3], t)} + d^\T r^{([1, 3], t)} = \alpha_{01} - \alpha_{02} + \alpha_{12} + 14 = 0, \nonumber \\
& \alpha^\T Q r^{([3, 1], t)} + d^\T r^{([3, 1], t)} = 2 \alpha_{03} - 2 \alpha_{02} = 0, \nonumber \\
& \alpha^\T Q r^{([3, 3], t)} + d^\T r^{([3, 3], t)} = \alpha_{03} - \alpha_{02} + \alpha_{23} + 14 = 0. \nonumber
\end{align} 

\section{\mytitle{Computational experiments}}
\label{section:experiments}

To evaluate the performance of the proposed corner Benders' cuts, we conducted computational experiments on VRPSD instances. Therefore, in the remainder of this section, we use the notation and concepts introduced in Section~\ref{subsection:vrpsd}. In particular, we use the sets~$X \subseteq \R^E$ and~$Y \subseteq \R^{\mathcal{A}}$ as defined in Section~\ref{subsection:vrpsd_formulation}. 

As a baseline, we implemented our own version of the state-of-the-art integer L-shaped algorithm by~\cite{parada2024disaggregated} (denoted \textsc{parada}). We then evaluated the algorithm's performance when the root node relaxation was strengthened by the following methods for separating Benders' cuts:
\begin{itemize}[leftmargin=*]
    \item \textsc{benders}:  Benders decomposition with the normalization proposed by~\cite{fischetti2010note};
    \item \textsc{lagrange}:  single Lagrangian cut derived from Theorem~\ref{thm:lagrangian_cut} (inspired by \cite{chen2024recovering}); and
    \item \textsc{corner}: corner Benders' cuts for~$\epifY[C]$, where~$C$ is obtained according to Proposition~\ref{prop:strong_corner}.
\end{itemize}
We do not show the results with the objective function cut, since it is always worse than \textsc{lagrange}. %

\subsection{\mytitle{Benchmark instances}}
\label{subsection:instances}

We tested the efficiency of our method on the instances proposed by~\cite{parada2024disaggregated} with~$|V| \in \{20, 30, 40, 50\}$ vertices, giving a total of 720 benchmark instances.
Three instances with 20 customers and~$k = 7$ are infeasible (\texttt{\small 20\_7\_0.95\_0}, \texttt{\small 20\_7\_0.95\_3} and \texttt{\small 20\_7\_0.95\_8}). Accordingly, we adjusted the number of vehicles to 8 for these instances.%
\footnote{The complete set of instances and their computational results are available on the website of one of the authors:~\url{https://sites.google.com/view/jfcote/}. The reported solutions for the three mentioned instances indeed use 8 vehicles.}
In every instance, for each customer~$v \in V_+$, customer demand~$q_v$ follows a normal distribution with variance and mean equal to~$\bar{q}_v$, meaning that these instances satisfy Assumption~\eqref{assumption:vrpsd}. To more accurately assess the performance of our method, we divided the instances into ``few vehicles'' or ``small~$k$'' instances ($k \in \{2, 3, 4\}$) and ``many vehicles'' or ``large~$k$'' instances ($k \in \{5, 6, 7\}$). This categorization was motivated by the observation that the algorithm \textsc{parada} already performs very well for small~$k$ instances, but it struggles for instances with large~$k$. For each instance, we generate the network~$\mathcal{N}$ according to the construction described in Section~\ref{subsection:vrpsd_formulation}. In Table~\ref{table:network_size} we report the average number of nodes and arcs in the resulting networks.

\begin{table}[htb!]
\centering
\footnotesize
\begin{tabular}{l@{\hskip 2em}rrrrrrrrrrrr}
\toprule
& \multicolumn{2}{c}{$k = 2$} & \multicolumn{2}{c}{$k = 3$} & \multicolumn{2}{c}{$k = 4$} & \multicolumn{2}{c}{$k = 5$} & \multicolumn{2}{c}{$k = 6$} & \multicolumn{2}{c}{$k = 7$} \\ 
\cmidrule(r){2-3} \cmidrule(r){4-5} \cmidrule(r){6-7} \cmidrule(r){8-9} \cmidrule(r){10-11} \cmidrule(r){12-13}
$|V|$ & $|\mathcal{V}|$ & $|\mathcal{A}|$ & $|\mathcal{V}|$ & $|\mathcal{A}|$ & $|\mathcal{V}|$ & $|\mathcal{A}|$ & $|\mathcal{V}|$ & $|\mathcal{A}|$ & $|\mathcal{V}|$ & $|\mathcal{A}|$ & $|\mathcal{V}|$ & $|\mathcal{A}|$ \\
\midrule
20 & 1023 & 17494 & 647 & 10367 & 445 & 6606 & 350 & 4739 & 280 & 3364 & 228 & 2426 \\
30 & 2477 & 67252 & 1563 & 40851 & 1131 & 28373 & 923 & 22074 & 715 & 16232 & 589 & 12595 \\
40 & 4567 & 169719 & 2879 & 104160 & 2153 & 75738 & 1740 & 59308 & 1378 & 45448 & 1141 & 36239 \\
50 & 6988 & 329636 & 4609 & 212992 & 3488 & 157730 & 2749 & 121479 & 2240 & 96592 & 1850 & 77707 \\
\bottomrule
\end{tabular}
\caption{Average size of the network~$\mathcal{N}$ for different values of~$k$.}
\label{table:network_size}
\end{table}

\subsection{\mytitle{Implementation details}}
\label{subsection:implementation}

All algorithms were executed in single-thread mode on a machine equipped with an Intel(R) Xeon(R) Gold 6142 CPU @ 2.60GHz processor. The code was implemented in C++, with Gurobi 12 as the LP/MIP solver and the Lemon library~\citep{lemon} for handling basic graph operations. The time limit for every run of the algorithms was set at 1 hour. Our implementation of the state-of-the-art algorithm of~\cite{parada2024disaggregated} is an exact branch-and-cut algorithm for the VRPSD that separates the rounded capacity inequalities (RCIs) of~\cite{laporte1983branch} and the integer L-shaped (ILS) cuts of~\cite{parada2024disaggregated}.%
\footnote{Appendix~\ref{appendix:parada} provides further details on the separation routines of~\cite{parada2024disaggregated}, identifies a minor mistake in their algorithm description, and compares in more detail the results of our implementation with those in the authors' online table.}
The formulation proposed by \cite{parada2024disaggregated} uses ``disaggregated'' second-stage variables, that is, for each customer~$v \in V_+$, the formulation has a variable~$\theta'_v$ that represents the expected second-stage cost paid at customer~$v$. \cite{parada2024disaggregated} propose different ILS cuts that guarantee that, at integer solutions, the sum~$\sum_{v \in V_+} \theta'_v$ correctly matches the expected second-stage cost of the first-stage solution. We refer the reader to the paper by \cite{parada2024disaggregated} for more details. Note that any Benders' cut~$\alpha^\T x + \alpha_0 \theta \geq \beta$ can be translated to the space of variables~$(x, \theta') \in \R^E \times \R^{V_+}$ as the inequality~$\alpha^\T x + \alpha_0 \, (\sum_{v \in V_+ } \theta' _v) \geq \beta$.

In our experiments, we observed that our implementation of the algorithm proposed by \cite{parada2024disaggregated} is competitive with the results reported in their online table: we can solve 5 more instances in the time limit and our implementation achieves root gaps that are, on average, approximately 1.5 times smaller. To achieve such stronger LP relaxation bounds, we start our solver by relaxing the integrality constraints and using a custom cutting-plane loop for separating RCIs and ILS cuts. This loop terminates under one of two conditions: (i) no additional cuts can be separated, or (ii) the bound fails to improve by at least \(10^{-3}\) over the last 10 iterations.  Once the loop stops, we enforce integrality on the~\(\{x_e\}_{e \in E}\) variables and solve the resulting MIP model with Gurobi, using callbacks to separate RCIs and ILS cuts during the solution process. Appendix~\ref{appendix:cutting_plane_loop} contains a pseudocode of our cutting-plane loop.

To evaluate the performance of \textsc{benders}, \textsc{lagrange}, and \textsc{corner}, we adapted our custom cutting-plane procedure to include the separation of the corresponding cuts. Appendix~\ref{appendix:separation_benders} provides further details on the separation of each type of Benders' cuts.

\subsection{\mytitle{Computational results}}
\label{subsection:experiments_results}

The results of our experiments are summarized in Figure~\ref{figure:experiments}. Loosely speaking, these figures show the cumulative distribution of the instances with respect to one of two metrics: \emph{gap} or \emph{time}.%
\footnote{The \emph{(remaining) gap} is computed as $100 \times (z_{\textrm{opt}} - z_{\textrm{cuts}}) / z_{\textrm{opt}}$, where
    $z_{\textrm{opt}}$ is the optimal value of the problem (known for our instances)
    and 
    $z_{\textrm{cuts}}$ is the bound after adding cuts to the LP relaxation.
}
Take Figure~\ref{figure:gap_small} as an example and let~$p = (p_1, p_2) \in \R^2$ be a point in the line of Figure~\ref{figure:gap_small} that corresponds to algorithm \textsc{corner}. Let~$\mathcal{I}$ be the set of all small~$k$ instances and let~$\hat{\mathcal{I}}$ be the number of these instances whose LP relaxation bound computed by \textsc{corner} has a gap of at most~$p_1$. The number~$p_2$ corresponds to the ratio~$|\hat{\mathcal{I}}| / |\mathcal{I}|$. Similarly, for Figure~\ref{figure:time_small} and a point~$p' = (p'_1, p'_2) \in \R^2$ in the line of algorithm \textsc{corner}, we have that~$p'_2 = |\hat{\mathcal{I}}'| / |\mathcal{I}|$, where~$\hat{\mathcal{I}}' \subseteq \mathcal{I}$ is the set of instances that \textsc{corner} solved the problem in a runtime of at most~$p'_1$.

\begin{figure}
    \centering
    \subfloat[Gap for small~$k$ instances.]{
        \includegraphics[width=0.44\linewidth]{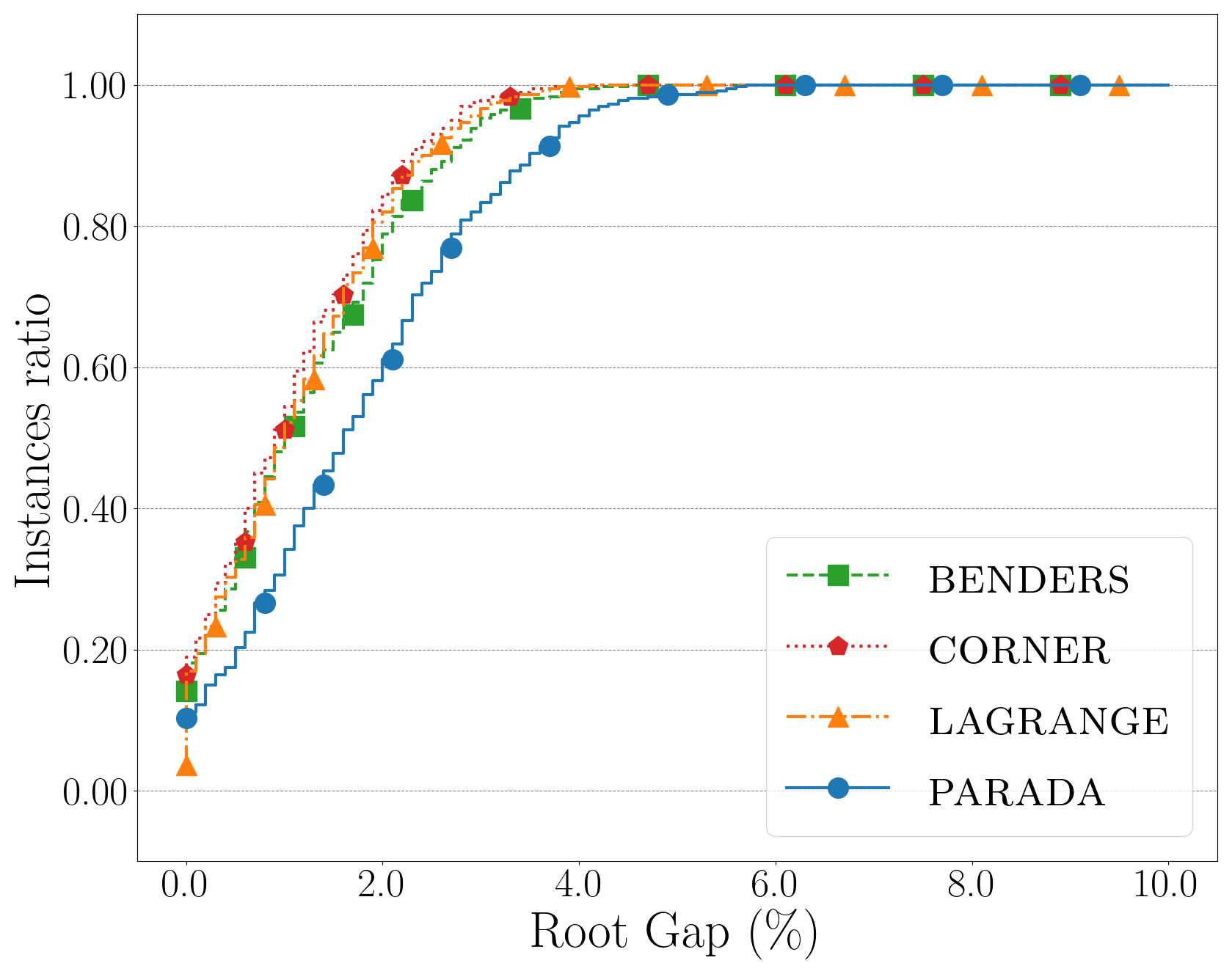}
        \label{figure:gap_small}
    }
    \hfill
    \subfloat[Time for small~$k$ instances.]{
        \includegraphics[width=0.44\linewidth]{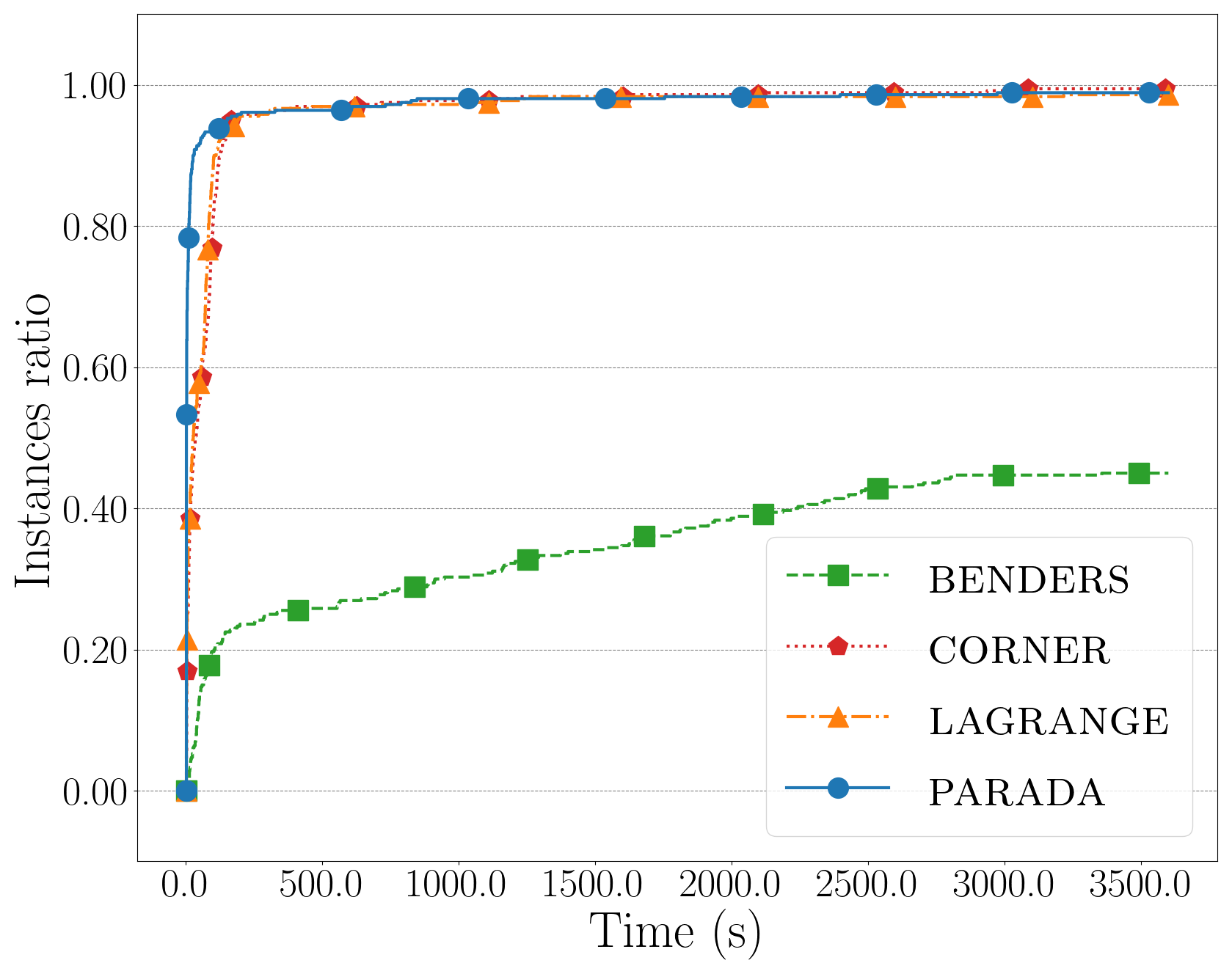}
        \label{figure:time_small}
    }
    \\[0.2cm]
    \subfloat[Gap for large~$k$ instances.]{
        \includegraphics[width=0.44\linewidth]{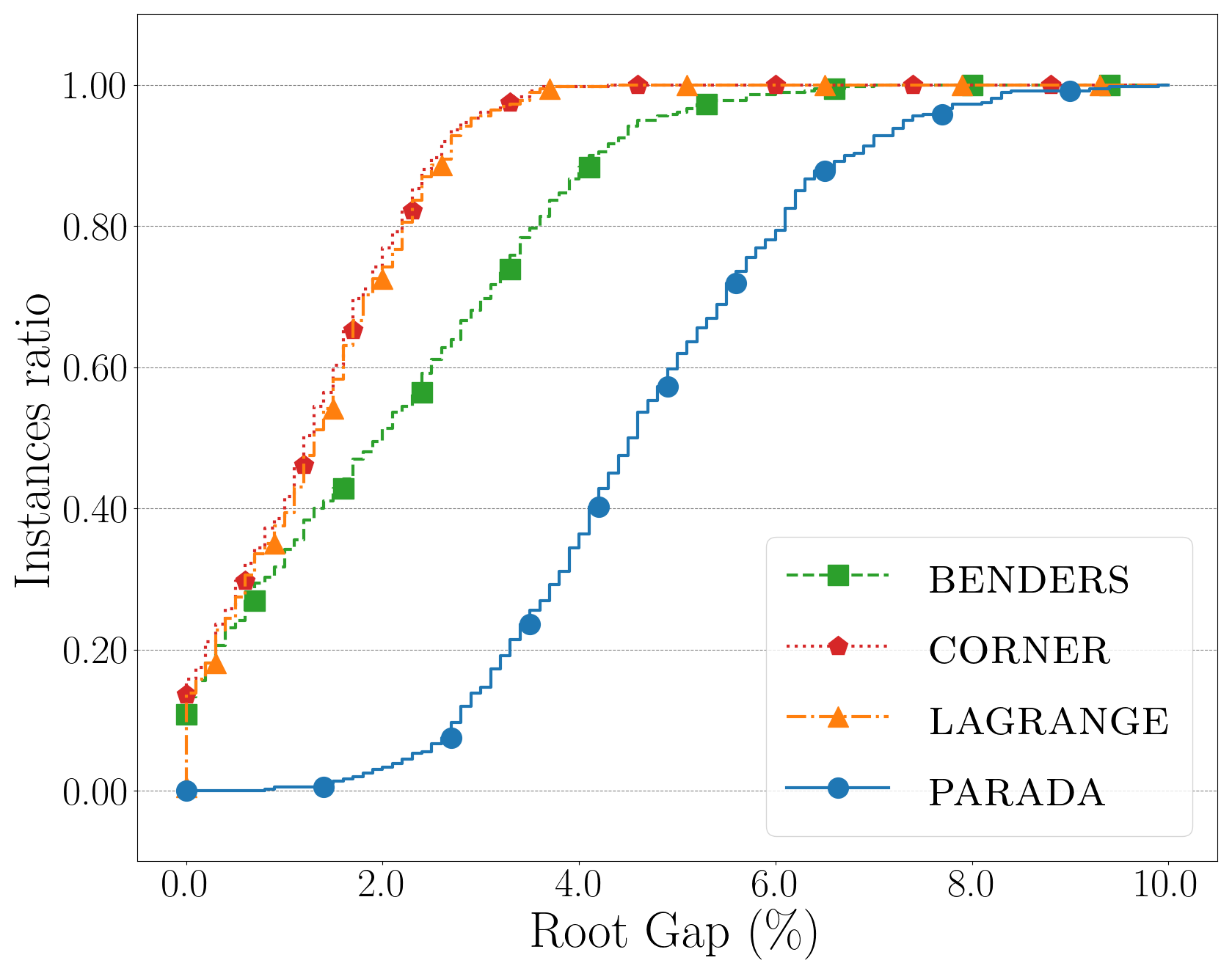}
        \label{figure:gap_large}
    }
    \hfill
    \subfloat[Time for large~$k$ instances.]{
        \includegraphics[width=0.44\linewidth]{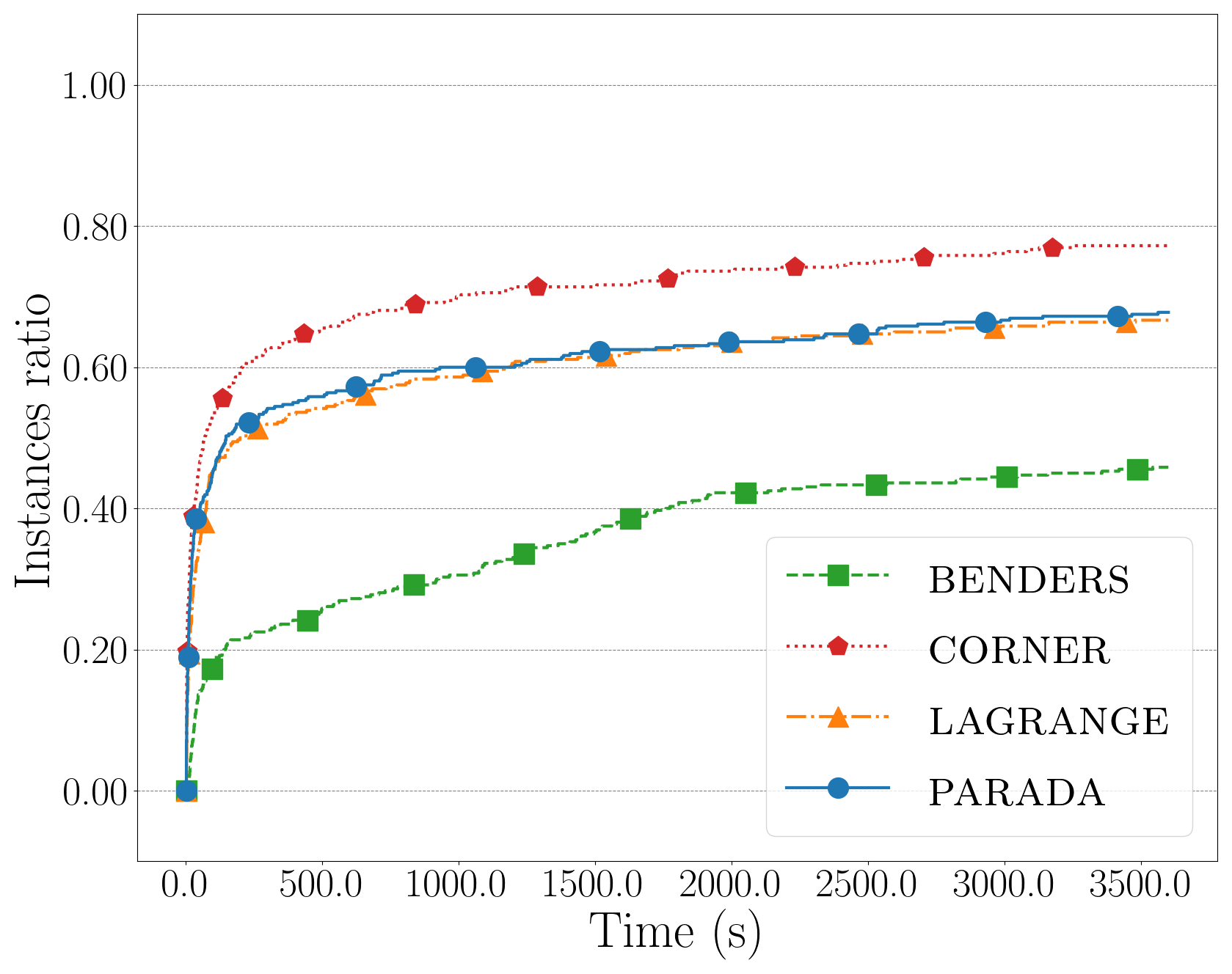}
        \label{figure:time_large}
    }
    \\[0.5cm]
    \caption{
        Computational results for instances of Parada et al.
    }
    \label{figure:experiments}
\end{figure}

On one hand, we see in Figure~\ref{figure:time_small} that, for small~$k$ instances, all the approaches (except \textsc{benders}) are fairly efficient and can solve almost all of the instances within the time limit. Moreover, Figure~\ref{figure:gap_small} shows that both \textsc{corner} and \textsc{lagrange} achieved slightly stronger root gaps than \textsc{benders}, and the execution time of both is competitive with \textsc{parada}. On the other hand, in Figures~\ref{figure:gap_large} and~\ref{figure:time_large}, we see that, for large~$k$ instances, none of the approaches can solve all of the instances; and only \textsc{corner} can solve more instances than \textsc{parada}. Looking at Figure~\ref{figure:gap_large}, we see again that both \textsc{corner} and \textsc{lagrange} achieved fairly strong root gaps, and Figure~\ref{figure:time_large} shows that algorithm \textsc{corner} has a significantly better overall execution time. Table~\ref{table:solved} shows the total number of instances solved by each algorithm within the time limit.

\begin{table}[htb!]
\centering
\begin{tabular}{rrrr}
\toprule
\textsc{parada} & \textsc{benders} & \textsc{lagrange} & \textsc{corner} \\ 
\midrule
600 & 327 & 595 & 636 \\ 
\bottomrule
\end{tabular}
\caption{Number of instances (out of 720) solved by each algorithm in the time limit of 1 hour.}
\label{table:solved}
\end{table}

\paragraph{\textbf{Number of branch-and-bound nodes.}}

The fact that \textsc{corner} is faster than \textsc{lagrange} might be because algorithm \textsc{corner} gives information of several facet-defining inequalities for~$\epifY[C]$, while the only information that \textsc{lagrange} has on the value function~$\fY$ is a single Lagrangian cut. To further investigate this matter, we show in Figure~\ref{figure:number_nodes} the number of branch-and-bound nodes explored by each method. In these figures, a point~$p = (p_1, p_2) \in \R^2$ in the line of an algorithm is so that~$p_2 = |\hat{\mathcal{I}}| / |\mathcal{I}|$, where~$\mathcal{I}$ is the set of all considered instances and~$\hat{\mathcal{I}} \subseteq \mathcal{I}$ is the set of instances to which the algorithm solved the problem to optimality by exploring at most~$p_1$ nodes of the branch-and-bound tree.

\begin{figure}[htb]
    \centering
    \subfloat[Nodes explored for small~$k$ instances.]{
        \includegraphics[width=0.44\linewidth]{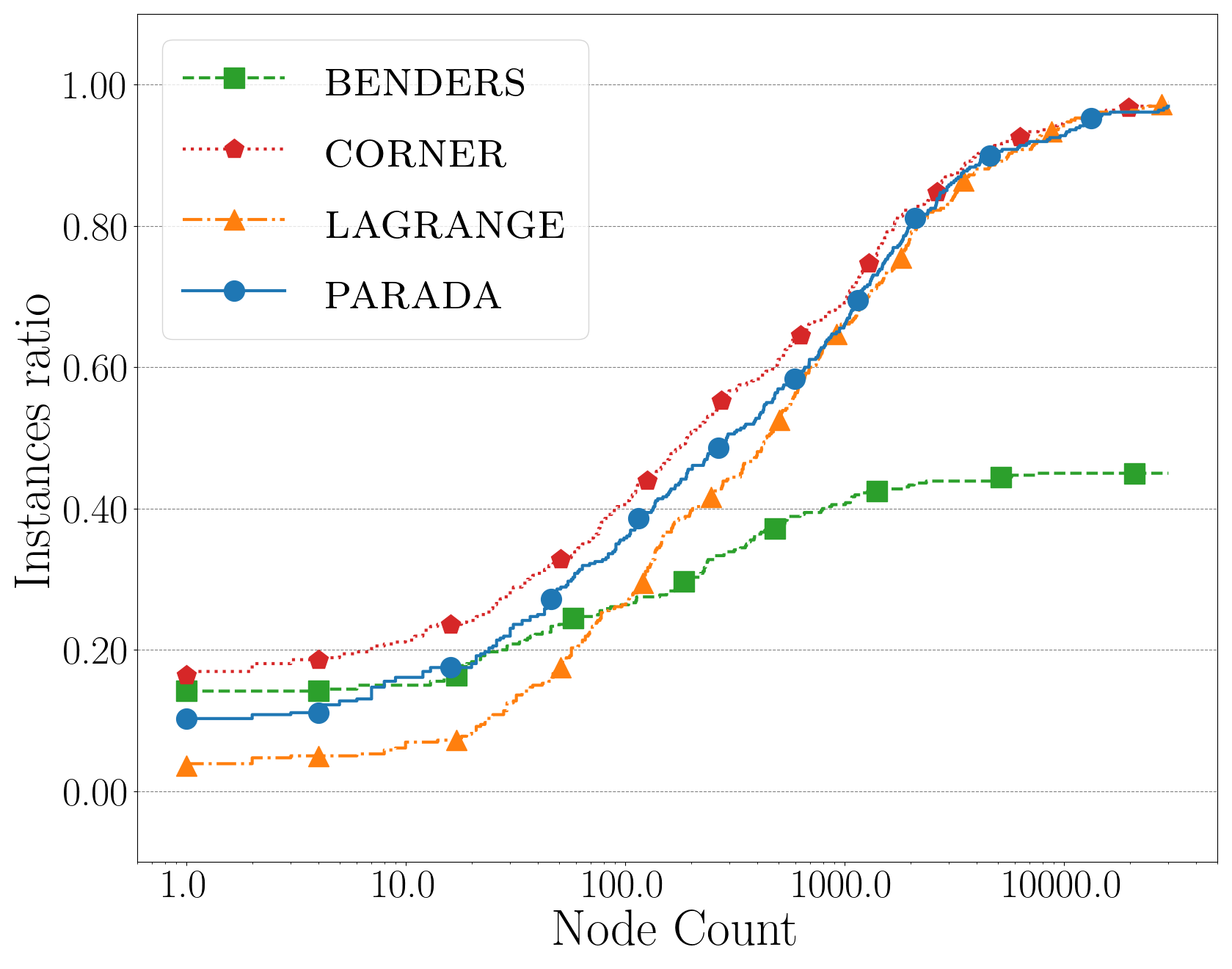}
        \label{figure:node_small}
    }
    \hfill
    \subfloat[Nodes explored for large~$k$ instances.]{
        \includegraphics[width=0.44\linewidth]{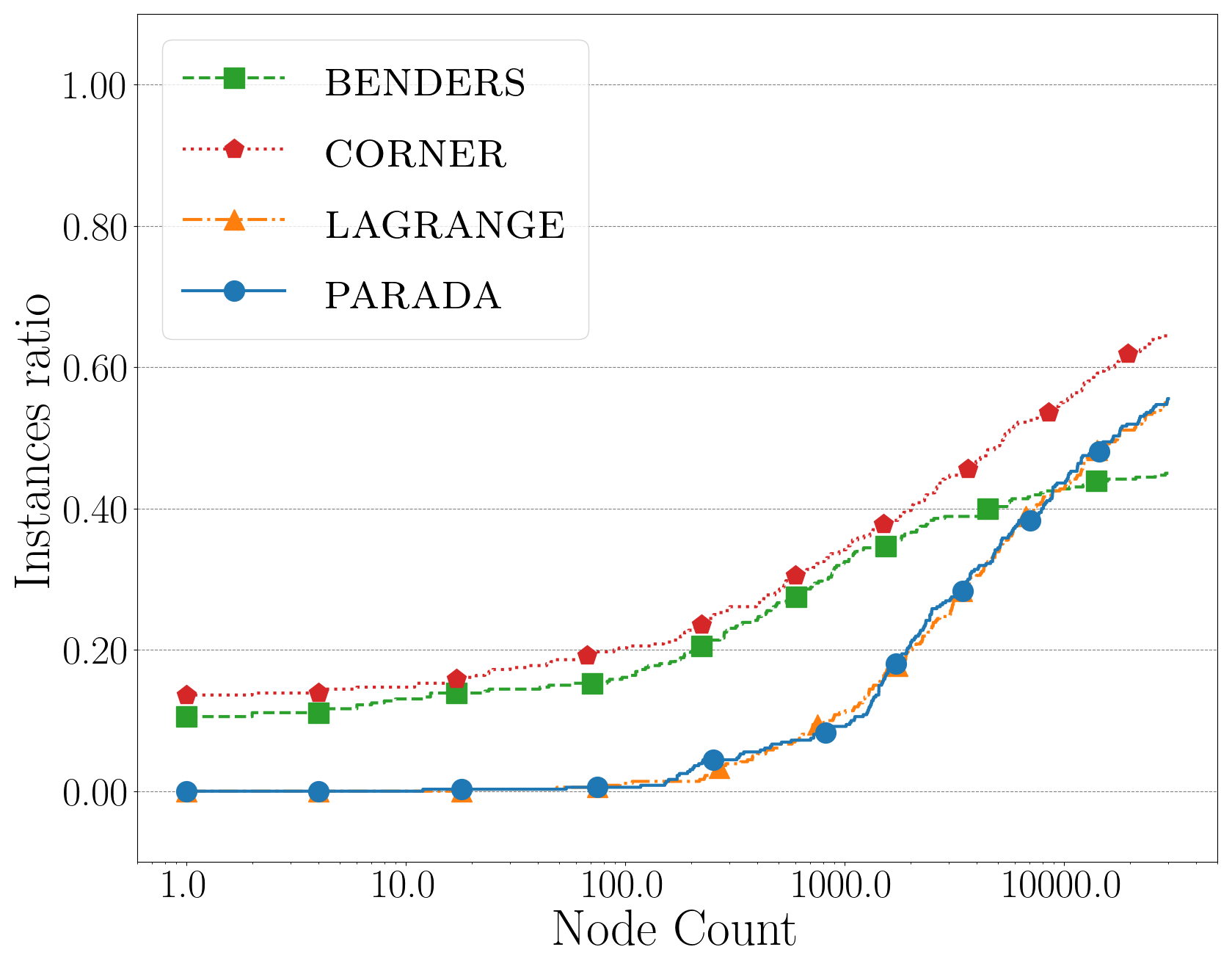}
        \label{figure:node_large}
    }
    \\[0.5cm]
    \caption{
        Results for the number of branch-and-bound nodes explored by the different algorithms.
    }
    \label{figure:number_nodes}
\end{figure}

Figure~\ref{figure:number_nodes} clearly shows that \textsc{corner} explores fewer branch-and-bound nodes than \textsc{lagrange}. Furthermore, \textsc{corner} solves approximately 16\% (respectively, 14\%) of all small~$k$ instances (respectively, large~$k$ instances) in the root node. This result suggests that corner Benders' cuts might be an efficient way of replacing the objective function cut, while the Lagrangian cut approach may still lead to high-dimensional optimal faces, which can be detrimental to the performance of the MIP solver.

\paragraph{\textbf{\mytitle{Root bound convergence.}}}

Next, we analyze the convergence speed of \tsc{corner} to the optimal bound~$z^*$ in Problem~\ref{problem_master}, comparing it with the other two Benders' cuts-based methods. Figure~\ref{figure:convergence} illustrates the improvement in the LP relaxation bound over time (we use a representative instance to showcase our findings, but similar behavior is observed across other instances). The horizontal line for \tsc{lagrange} indicates that the optimal bound for this instance is~$z^* = 1038.07$, which, according to Theorem~\ref{thm:lagrangian_cut}, can be achieved with a single Lagrangian cut. Observe that in 7 seconds, algorithm \tsc{corner} reaches the optimal bound, while \tsc{benders} fails to even reach a bound of~$1015$. For this same instance, Appendix~\ref{appendix:experiments_separation_rays} reports the time to separate each Benders' cut and the number of ray inequalities generated by Algorithm~\tsc{SolvePolar}.

\begin{figure}
    \centering
    \includegraphics[scale=0.22]{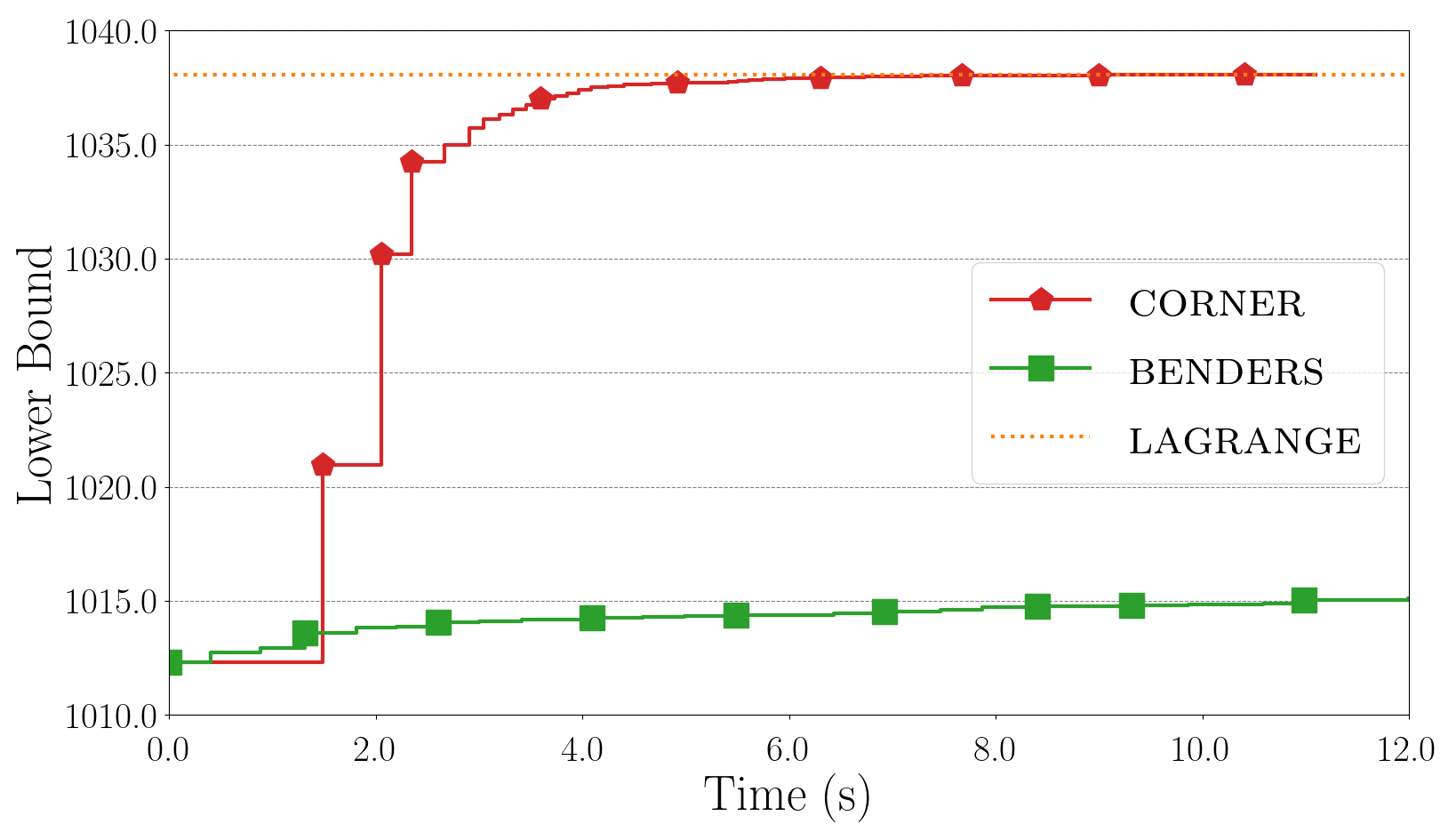}
    \caption{Bound improvement for instance ``30\_7\_0.95\_1.00\_8'' of~\cite{parada2024disaggregated} ($|V| = 30$ and~$k = 7$).}
    \label{figure:convergence}
\end{figure}

Table~\ref{table:number_cuts} presents the average number of Benders' cuts separated (recall that corner Benders' cuts are Benders' cuts) and the average root node solving time (which is capped at 1 hour if the algorithm does not solve the root node within the time limit). Observe that \tsc{benders} separates significantly more cuts than \tsc{corner} and takes much more time to converge, strengthening our point that \tsc{benders} has worse convergence to the optimal bound. Furthermore, as previously shown, while \tsc{lagrange} reaches the optimal bound with a single cut, it is generally computationally less efficient than \tsc{corner} for solving the problem to integrality.

\begin{table}[htb!]
\centering
\begin{tabular}{r@{\hskip 2em}rrrrr}
\toprule
& \multicolumn{2}{c}{Benders' cuts} & & \multicolumn{2}{c}{Root Time (s)} \\
\cmidrule(r){2-3} \cmidrule(r){5-6}
& \tsc{benders} & \tsc{corner} && \tsc{benders} & \tsc{corner} \\
\midrule
small~$k$ & 143.88 & 20.31 && 2315.67 & 44.31 \\
large~$k$ & 637.98 & 75.43 && 2217.13 & 42.65 \\
\bottomrule
\end{tabular}
\caption{Average number of Benders' cuts separated and root node solving time for algorithms \tsc{benders} and \tsc{corner}.}
\label{table:number_cuts}
\end{table}

\section{Conclusion}

In this work, we have developed a new way to generate Benders' cuts by exploiting basis information on the LP solved by the Benders' cut generating subproblem. 
When applied to instances of the VRPSD, our approach improved the performance of a state-of-the-art problem-specific algorithm found in the literature.

We believe the idea of exploring basis information to generate Benders' cuts shows much promise and future work includes testing this approach in other important problems, as well as considering how the consideration of integrality on the Benders' subproblem variables may affect the cut-generating process using such basis information.

\bibliographystyle{plainnat}
\bibliography{bibliography}

\renewcommand{\theHsection}{A\arabic{section}}
\begin{APPENDICES}
\section{Proof of Theorem~\ref{thm:polar}}
\label{appendix:proof_polar}

To facilitate the reading, we repeat the theorem statement.
\polar*
\begin{proof}
We start with item (b). Suppose that Problem~\eqref{problem_polar} is unbounded, so we have a certificate~$(\alpha, \alpha_0) \in \recc(\revpolar)$. Since~$0\in \revpolar$, for all~$\mu \in \R_+$, we have that~$\mu  (\alpha, \alpha_0) \in \revpolar$. This implies that
\begin{equation}
    \alpha^\T (w^* - w') + \alpha_0  (\theta^* - \theta') \geq 0 ~ \iff ~ \alpha^\T w^* + \alpha_0  \theta^* \geq \alpha^\T w' + \alpha_0  \theta'.
    \label{ineq:proof_polar2}
\end{equation}
The inequality above is valid for the unique extreme point of~$\epi(\hyperref[definition:value_function]{f_C})$, and since~$(\alpha, \alpha_0)$ satisfy all ray inequalities of~$\revpolar$ (and~$\alpha_0 \geq 0$), we know that~\eqref{ineq:proof_polar2} is a valid inequality for~$\epi(\hyperref[definition:value_function]{f_C})$. Moreover,~\eqref{ineq:proof_polar2} is tight at~$(w', \theta')$, which is a point in the relative interior of~$\epi(\hyperref[definition:value_function]{f_C})$. This shows that~\eqref{ineq:proof_polar2} holds with equality for all points in~$\epi(\hyperref[definition:value_function]{f_C})$.

To prove (a), we assume by contradiction that~$z^\circ \geq -1$ and~$(\bar{w}, \bar{\theta}) \notin \epi(\hyperref[definition:value_function]{f_C})$. Using the separating hyperplane theorem, we have a Fenchel cut~$\alpha^\T w + \alpha_0 \theta \geq \sigma_{\epi(\hyperref[definition:value_function]{f_C})}(\alpha, \alpha_0)$ (with~$\alpha_0 \geq 0$) that is valid for~$\epi(\hyperref[definition:value_function]{f_C})$ but is violated by~$(\bar{w}, \bar{\theta})$. Suppose first that~$\alpha^\T w' + \alpha_0 \theta' = \sigma_{\epi(\hyperref[definition:value_function]{f_C})}(\alpha, \alpha_0)$, meaning that the Fenchel cut associated with~$(\alpha, \alpha_0)$ defines an implicit equality of~$\epi(\hyperref[definition:value_function]{f_C})$. We can then follow the reasoning for item (b) to learn that~$(\alpha, \alpha_0) \in \recc(\revpolar)$; and as~$\alpha^\T \bar{w} + \alpha_0 \bar{\theta} < \sigma_{\epi(\hyperref[definition:value_function]{f_C})}(\alpha, \alpha_0) = \alpha^\T w' + \alpha_0 \theta'$, we have that~$z^\circ = - \infty$, a contradiction. So we can safely assume that~$\alpha^\T w' + \alpha_0 \theta' > \sigma_{\epi(\hyperref[definition:value_function]{f_C})}(\alpha, \alpha_0)$. Now subtract~$\alpha^\T w' + \alpha_0  \theta'$ from both sides of~$\alpha^\T w + \alpha_0 \theta \geq \sigma_{\epi(\hyperref[definition:value_function]{f_C})}(\alpha, \alpha_0)$ to get
\begin{equation}
    \alpha^\T (w - w') + \alpha_0  (\theta - \theta') \geq \beta \defeq \sigma_{\epi(\hyperref[definition:value_function]{f_C})}(\alpha, \alpha_0) - \alpha^\T w' - \alpha_0  \theta',
    \label{ineq:proof_polar1}
\end{equation}
where~$\beta < 0$. Dividing both sides of~\eqref{ineq:proof_polar1} by~$|\beta|$ and setting~$\alpha' = (1 / |\beta|) \alpha$ and~$\alpha'_0 = (\alpha_0 / |\beta|)$ yields the inequality~$\alpha' (w - w') + \alpha'_0 (\theta - \theta') \geq -1$, which implies that~$z^\circ < -1$, a contradiction. To see this, note that since the Fenchel cut defined by~$(\alpha, \alpha_0)$ is violated by~$(\bar{w}, \bar{\theta})$, we know that~$\alpha' (\bar{w} - w') + \alpha'_0 (\bar{\theta} - \theta') < -1$. Moreover, as~$(\alpha')^\T w + \alpha'_0 \theta \geq -1$ is a valid inequality for~$E = \epi(\hyperref[definition:value_function]{f_C}) - (w', \theta')$, we have that~$(\alpha', \alpha'_0)$ belongs to~$\revpolar$.

Lastly, suppose that we have~$(\alpha, \alpha_0)$ such that~$\alpha^\T (\bar{w} - w') + \alpha_0 (\bar{\theta} - \theta') = z^\circ < -1$ and~$(\alpha, \alpha_0)$ is a vertex of~$\revpolar$. By translation, it suffices to show that~$\alpha^\T w + \alpha_0 \theta \geq -1$ is a facet of~$E$. Let a minimal representation for~$E$ be~
\begin{equation}
    E =
    \left\{ (\alpha, \alpha_0) \in \R^{p + 1} :~
    \begin{aligned}
    & (\rho^j)^\T x + \rho^j_0 \theta \geq -1, & \forall j \in J \\
    & (\rho^j)^\T x + \rho^j_0 \theta = 0, & \forall j \in J'
    \end{aligned}
    \right\},
    \label{minimal_repr}
\end{equation}
where~$J$ and~$J'$ are index sets. For all~$j \in J \cup J'$, we know that~$(\rho^j, \rho^j_0) \in \revpolar$. Additionally, as~$\alpha^\T w + \alpha_0 \theta \geq -1$ is a valid inequality for~$E$, there exists~$\mu \geq 0$ and~$\mu' \geq 0$ such that~$(\alpha, \alpha_0) = \sum_{j \in J} \mu_j (\rho^j, \rho^j_0) + \sum_{j \in J'} \mu'_j (\rho^j, \rho^j_0)$ and~$-1 = \sum_{j \in J} \mu_j (-1) + \sum_{j \in J'} \mu'_j (0)$ (this follows from the fact that~\eqref{minimal_repr} is a minimal representation, see Section 3.9 of~\cite{ipbook}). Since~$(\alpha, \alpha_0)$ is an extreme point of~$\revpolar$, it follows that~$|J| = 1$ and~$\alpha^\T w + \alpha_0 \theta \geq -1$ defines a facet of~$E$.
\end{proof}

\section{\mytitle{Corner Benders' cuts as restricted dual subproblems}}
\label{appendix:corner_restricted_dual}

Consider the setup in Remark~\ref{remark:simplex} and assume that~$C$ is a corner associated with a set of basic variables~$B$ (and~$N = [m] \setminus B$). In this case, it is known that we can write~$C = \{y \in \R^m \suchthat Ay = b, y_B \in \R^B, y_N \in \R^N_+\}$ (see Chapter 6.1 of~\cite{ipbook}), and therefore we can check if~$(\bar{w}, \bar{\theta})$ lies in~$\epifY[C]$ by solving the problem
\begin{equation}
\begin{aligned} 
\min_{y} ~~& 0 & \\
\textsuchthat{}~& -d^{\T}_B y_B - d^{\T}_N y_N \geq - \bar{\theta}, & (\alpha_0) \\
& Q_B y_B + Q_N y_N = \bar{w}, & (\alpha) \\
& A_B y_B + A_N y_N = b, & (\nu) \\
& y_B \in \R^{B}, y_N \in \R^{N}_+.
\end{aligned}
\label{problem:cone_benders_1}
\end{equation}

\noindent
By LP duality, we have that~$(\bar{w}, \bar{\theta}) \in \epifY[C]$ if and only if the optimal value of the optimization problem below is at most zero.
\begin{subequations}
\begin{align} 
\max_{\nu,\alpha,\alpha_0} ~~& \nu^\T b + \alpha^\T \bar{w} - \alpha_0 \bar{\theta} & \nonumber \\
\textsuchthat{}~& -\alpha_0 d_B + Q^\T_B \alpha + A^\T_B \nu = 0, \label{problem:cone_benders_2:ineq1} \\
& - \alpha_0 d_N + Q^\T_N \alpha + A^\T_N \nu \leq 0, \label{problem:cone_benders_2:ineq2} \\
& \alpha_0 \geq 0.
\end{align}
\label{problem:cone_benders_2}
\end{subequations}

Observe that if variables~$y_B$ were nonnegative in Problem~\eqref{problem:cone_benders_1}, then~\eqref{problem:cone_benders_2:ineq1} would be inequalities instead of equalities. In this sense, relaxing~$\epifY$ to~$\epifY[C]$ is equivalent to constraining the Benders' dual subproblem according to~\eqref{problem:cone_benders_2:ineq1}. This idea of restricting a cut-generating program to separate inequalities more efficiently was explored by~\cite{chen2022generating} in the context of stochastic programs with integer subproblems. Our discussion here shows that corner Benders' cuts can also be interpreted from a similar lens, except that here we constrain the cut-generating program according to the given basis~$B$.

\section{\mytitle{Lagrangian dual problem for block-diagonal matrices}}
\label{appendix:block_diagonal}

Suppose Problem~\eqref{problem:basic} can be written as
\begin{equation}
\begin{array}{lllll}
\min & c^\T x + \sum_{i \in [t]} (d^i)^\T y^i \\[0.5em]
\textsuchthat{} & T^i x + Q^i y^i = h^i, & \forall i \in [t], & \quad (\alpha^i) \\
& x \in X, \\
& y^i \in Y^i, & \forall i \in [t].
\end{array}
\label{problem:block_diagonal}
\end{equation}
Reformulating Problem~\eqref{problem:block_diagonal} in an epigraphical form yields
\begin{equation}
    \min_{x,\theta} \left\{c^\T x + \sum_{i \in [t]} \theta_i \suchthat x \in X,\, (h^i - T^i x, \theta^i) \in \epi(f^i),~\forall i \in [t] \right\},
    \label{problem:block_diagonal_epigraph}
\end{equation}
where~$f^i(w) = \min \{(d^i)^\T y^i : Q^i y^i = w, y^i \in Y^i\}$ and~$\epi(f^i) = \{(w, \theta) : \theta \geq f^i(w)\}$, for all~$i \in [t]$. 

Let~$\mathcal{F}'$ be the feasible region of Problem~\eqref{problem:block_diagonal_epigraph}. We can easily adapt Theorem~\ref{thm:lagrangian_cut} to Problem~\eqref{problem:block_diagonal_epigraph}. Indeed, given~$\rho \in \R^n$ and~$\rho^1_0, \ldots, \rho^t_0 \in \R_+$, one can check that if~$\hat{\alpha}^1, \ldots, \hat{\alpha}^t$ is optimal for
\begin{equation}
    \max_{\alpha^1, \ldots, \alpha^t} \left\{ - \sum_{i \in [t]} (\alpha^i)^\T h^i + \sigma_{X}\left(\rho + \sum_{i \in [t]} (T^i)^\T \alpha^i \right) + \sum_{i \in [t]} \sigma_{Y^i} \left((Q^i)^\T \alpha^i + \rho^i_0 \, d^i \right)\right\},
    \label{problem:lagrangian_appendix}
\end{equation}
then the Benders' cuts~$(\hat{\alpha}^i)^\T (h^i - T^i x) + \rho^i_0 \,  \theta^i \geq \sigma_{\epi(f^i)}(\hat{\alpha}^i, \rho^i_0)$, for~$i \in [t]$, imply the inequality~$\rho^\T x + \sum_{i \in [t]} \rho^i_0 \theta \geq \sigma_{\mathcal{F}'}(\rho, \rho^1_0, \ldots, \rho^t_0)$.  Proposition~\ref{prop:strong_corner} can also be adapted to this block-diagonal case, yielding one corner~$C^i$ for each block of constraints~$i \in [t]$.

\section{\mytitle{Full cutting-plane loop for the VRPSD}}
\label{appendix:cutting_plane_loop}

In the following algorithm, we denote by \tsc{(VRPSD-BASE)} the linear program below.
\begin{equation}
    \begin{aligned}
    &\min_{x,\theta} && c^\T x + \sum_{v \in V_+} \theta'_v & \\
    & \textsuchthat{} && x(\delta(0)) = 2k, & \\
    &&& x(\delta(v)) = 2, & \forall v \in V_+, \\
    &&& x_e \leq 1, & \forall e \in E \setminus \delta(0), \\
    &&& x \in [0, 2]^E, \theta' \in \R^{V_+}_+. &
    \end{aligned}
    \tag{\tsc{VRPSD-BASE}}
    \label{formulation:vrpsd_algorithm}
\end{equation}

\begin{algorithm}[H]
\small
\begin{algorithmic}[1]
\setlength{\itemsep}{-0.2em}
\Procedure {VRPSD-CuttingPlane}{$\tsc{algorithm}$}
\If {$\tsc{algorithm} = \textsc{lagrange}$~\textbf{or}~$\tsc{algorithm} = \textsc{corner}$}
    \State {Get~$\hat{\alpha}$ optimal for Problem~\eqref{problem:lagrangian1} by solving Problem~\eqref{problem:lagrangian_flow}.}
    \State {Get a shortest-path tree~$\mathcal{T}$ with respect to weights~$\{d'_a = d_a + \hat{\alpha}_{\edge(a)}\}_{a \in \mathcal{N}}$.}
    \State {Set~$C = \{y^*\} + \cone(R)$ to be the corner associated with~$\mathcal{T}$.}
    \State {Set~$y'$ to be a point in the relative interior of~$Y$.}
    \State {$(w', \theta') \gets (Q y', d^\T y' + 1)$}
\EndIf
\State {$\tilde{\mathcal{R}} \gets \emptyset$}
\State {Initialize (\textsc{VRPSD-LP}) as the linear program~\eqref{formulation:vrpsd_algorithm}.}
\If {$\tsc{algorithm} = \textsc{lagrange}$}
    \State {Add Lagrangian Cut~$\hat{\alpha}^\T x + \theta \geq \sigma_Y(Q^\T \hat{\alpha} + d)$ to (\textsc{VRPSD-LP}).}
\EndIf
\State {$\tsc{separated} \gets \tsc{true}$}
\State {$\tsc{prevBound} \gets -\infty$}
\State {$\tsc{count} \gets 1$}
\While {$\tsc{separated}$~\textbf{and}~$\tsc{count} \leq 10$}
    \State {$\tsc{separated} \gets \tsc{false}$}
    \State {Solve~(\textsc{VRPSD-LP}) to get a candidate solution~$(\bar{x}, \bar{\theta}')$.}
    \State {$\bar{z} \gets c^\T \bar{x} + \sum_{v \in V_+} \bar{\theta}'_v$}
    \If {$\bar{z} - \tsc{prevBound} \leq 10^{-3}$}
        \State {$\tsc{count} \gets \tsc{count} + 1$}
    \Else
        \State {$\tsc{count} \gets 1$}
    \EndIf
    \State {$\tsc{prevBound} \gets \bar{z}$}
    \If {Heuristics found violated RCI's or ILS cuts} \label{line:algorithm_rci_ils}
        \State {Add violated cuts to (\textsc{VRPSD-LP}).}
        \State {$\tsc{separated} \gets \tsc{true}$}
    \ElsIf {$\tsc{algorithm} = \textsc{benders}$}
        \If {Problem~\eqref{problem:fischetti} is feasible}
            \State {Let~$\bar{\alpha}$ and~$\bar{\alpha}_0$ be optimal for Problem~\eqref{problem:fischetti}.}
            \State {Add inequality~$\bar{\alpha}_0 \, (\sum_{v \in V_+} \theta'_v) \geq \bar{\alpha}^\T x$ to (\textsc{VRPSD-LP}).}
            \State {$\tsc{separated} \gets \tsc{true}$}
        \EndIf
    \ElsIf {$\tsc{algorithm} = \textsc{corner}$}
        \State{$((\rho, \rho_0, \beta), z^\circ, \tilde{R}) \gets \tsc{SolvePolar}((\bar{x}, \bar{\theta}), (Qy', d^\T ), y^*, R, \tilde{R})$}
        \If {$z^\circ < -1$}
            \State {$\tsc{separated} \gets \tsc{true}$}
            \State {Add inequality~$\rho^\T x + \rho_0 \, (\sum_{v \in V_+} \theta'_v) \geq \beta$ to (\textsc{VRPSD-LP}).}
            \If {$z^\circ = -\infty$}
                \State {Add inequality~$\rho^\T x + \rho_0 \, (\sum_{v \in V_+} \theta'_v) \leq \beta$ to (\textsc{VRPSD-LP}).}
            \EndIf
        \EndIf
    \EndIf
\EndWhile

\State {\textbf{return} (\textsc{VRPSD-LP})}
\EndProcedure
\end{algorithmic}
\caption{\mytitle{Cutting-plane loop for the VRPSD}}
\label{algorithm:vrpsd}
\end{algorithm}

\section{\mytitle{Separation of different Benders' cuts}}
\label{appendix:separation_benders}

We now describe how the different Benders’ cuts evaluated in our computational experiments are separated.

\paragraph{\mytitle{Benders' decomposition.}}

To separate optimality/feasibility cuts in the standard Benders decomposition approach, we used the well-known normalization proposed by~\cite{fischetti2010note}. This is one of the normalizations used in CPLEX~\citep{bonami2020implementing} (another option is normalizing the righthand side as in \cite{conforti2019facet}), and it is used as a baseline approach for recently proposed Benders' cuts selection methods~\citep{seo2022closest, brandenberg2021refined, hosseini2024deepest}. 

Let~$(\bar{x}, \bar{\theta}') \in \R^E \times \R^{V_+}$ be a candidate solution. Set~$\bar{\theta} = \sum_{v \in V_+} \bar{\theta}'_v$ and consider the flow-based formulation described in Section~\ref{subsection:vrpsd_formulation}. Our Benders' decomposition implementation separates~$(\bar{x}, \bar{\theta})$ from~$\epifY$ by solving the following optimization problem:
\begin{subequations}
\label{problem:fischetti}
\begin{align} 
\max_{\alpha,\alpha_0,\nu} ~~& \alpha^\T \bar{x} - k \nu_s + k \nu_t - \alpha_0 \bar{\theta} & \nonumber \\
\textsuchthat{} & \nu_{\mathcal{S}_1} - \nu_{\mathcal{S}_2} + \alpha_{\edge(\mathcal{S}_1 \mathcal{S}_2)} \leq \alpha_0 \, d_{\mathcal{S}_1 \mathcal{S}_2}, & \forall \mathcal{S}_1 \mathcal{S}_2 \in \mathcal{A}, \\
& \| \alpha \|_1 + \alpha_0 \leq 1, & \\ 
& \alpha_0 \geq 0.
\end{align}
\end{subequations}

\paragraph{\mytitle{Lagrangian cut}.}

To solve the Lagrangian dual problem in Proposition~\ref{prop:strong_corner}, we substitute the~$x_e$ variables     %
    according to the linking constraints
    $\sum_{a \in \mathcal{A} : \edge(a) = e} y_a = x_e$
    in Problem~\eqref{problem:flow_cl},
and we solve the following flow-based formulation.
\begin{subequations}
\label{problem:lagrangian_flow} %
\begin{align} 
\min_{y} ~~& \sum_{a \in \mathcal{A}} (c_{\edge(a)} + d_a) \, y_a & \nonumber \\
\textsuchthat{} & y(\delta^-(\mathcal{S})) - y(\delta^+(\mathcal{S})) = k \cdot \allones_t - k \cdot \allones_s, & \forall \mathcal{S} \in \mathcal{V}, &&~~(\nu_{\mathcal{S}}) \label{lagrangian_flow:conserv} \\
& \sum_{e \in \delta(v)} \left( \sum_{a \in \mathcal{A} : \edge(a) = e} y_a \right) = 2, & \forall v \in V_+, &&~~(\gamma_{\{v\}}) \label{lagrangian_flow:degree} \\
& \sum_{e \in \delta(S)} \left( \sum_{a \in \mathcal{A} : \edge(a) = e} y_a \right) \geq 2 k(S), & \forall S \subseteq V_+, |S| \geq 2, &&~~(\gamma_S) \label{lagrangian_flow:rci} \\
& y_a \geq 0, & \forall a \in \mathcal{A},
\end{align}
\end{subequations}
where inequalities~\eqref{lagrangian_flow:rci} are separated heuristically with the CVRPSEP package~\citep{Lysgaard2004}. Let~$\bar{\nu}$ and~$\bar{\gamma}$ be optimal dual variables associated with Formulation~\ref{problem:lagrangian_flow}. For each edge~$e \in E$, we set~$\hat{\alpha}_e = c_e - \sum_{\emptyset \subset S \subseteq V_+} \bar{\gamma}_S$, and we prove in Appendix~\ref{appendix:solve_lagrangian} that~$\hat{\alpha}$ is optimal for Problem~\ref{problem:lagrangian1}.

\paragraph{\mytitle{Corner Benders' cuts}.}

The first step in generating corner Benders' cuts is to obtain~$\hat{\alpha}$ optimal for Problem~\ref{problem:lagrangian1} by solving Problem~\ref{problem:lagrangian_flow}, as described earlier. We then apply Proposition~\ref{prop:strong_corner} to find a corner that is optimal with respect to~$\sigma_Y(Q^\T \hat{\alpha} + d)$. As illustrated in Section~\ref{subsection:corner_example}, we obtain this corner by setting weights~$d'_a = d_a + \hat{\alpha}_{\edge(a)}$, for each~$a \in \mathcal{A}$, and then using these weights to find a shortest path tree~$\mathcal{T}$ rooted at~$s$. To obtain the interior point $y'$ in Algorithm~\ref{algorithm:separate_corner}, we optimize a zero-objective function over the set $Y$ using Gurobi's barrier algorithm.

In our implementation of Algorithm~\ref{algorithm:solve_polar}, given a candidate solution~$(\bar{\alpha}, \bar{\alpha}_0)$ to \cglp, we first loop through all arcs in~$\mathcal{A} \setminus \mathcal{T}$ to construct mappings (implemented as arrays)~$M_1 : E \to \mathcal{A}$ and~$M_2  : E \to \Q$ such that, for each edge~$e \in E$,~$M_1(e)$ is an arc that belongs to~$\argmin \left\{ \bar{\alpha}^\T (Q r^a) + \bar{\alpha}_0 (d^\T r^a) : \edge(a) = e,~a \in \mathcal{A} \setminus \mathcal{T} \right\}$ (or~$\emptyset$ if no such edge exists) and~$M_2(e)$ is the value of~$\bar{\alpha}^\T (Q r^{M_1(e)}) + \bar{\alpha}_0 (d^\T r^{M_1(e)})$ (or~$+\infty$ if~$M_1(e) = \emptyset$). (Recall that we defined in Section~\ref{subsection:corner_example} that~$r^a = \mathbbm{1}(C^a)$, for each~$a \in \mathcal{A} \setminus \mathcal{T}$.) Next, we create an array~$L$ of edges~$e$ that is in nondecreasing order of~$M_2(e)$. We then traverse over the first~$\min\{|L|, 100\}$ edges~$e$ in~$L$, and if~$M_2(e) < 0$, we add the inequality~$\bar{\alpha}^\T (Q r^{M_1(e)}) + \bar{\alpha}_0 (d^\T r^{M_1(e)}) \geq 0$ to~\cglp.

In Appendix~\ref{appendix:corner_network}, we also show how we further explore the special structure of corners arising from network flow polytopes to separate the ray inequalities faster. Although this acceleration is not essential for the effectiveness of corner Benders' cuts, we use it because it can lead to performance improvements. Our preliminary experiments in the instances with large~$k$ and~$|V| = 40$ indicate that using this acceleration can reduce the execution time of the root cutting-plane loop by an average factor of two, translating to a reduction of 25 seconds in the average runtime.

\section{\mytitle{Using network flow structure to accelerate the separation of corner Benders' cuts}}
\label{appendix:corner_network}

Consider the separation of ray inequalities in Algorithm~\ref{algorithm:solve_polar}, specifically the loop in lines~\ref{alg_solve_polar:begin_for}-\ref{alg_solve_polar:end_for}. Since matrix~$Q$ has~$m$ columns and~$p$ rows, each iteration of this loop has a time complexity of~$\mathcal{O}(|R| m p)$. In the specific case where~$Y$ is a network flow polytope associated with a network~$\mathcal{N} = (\mathcal{V}, \mathcal{A})$ and~$R$ represents the cycles described in Section~\ref{subsection:corner_example}, this time complexity simplifies to~$\mathcal{O}(|\mathcal{A}|^2 p)$. In contrast, Algorithm~\ref{algorithm:network_separation} below detects violated ray inequalities in~$\mathcal{O}(|\mathcal{A}| p)$ time. Furthermore, in the particular case of the VRPSD formulation shown in Section~\ref{subsection:vrpsd_formulation}, the product~$\bar{\alpha}^\T Q_a$ simplifies to~$\bar{\alpha}_{\edge(a)}$, which can be computed in~$\mathcal{O}(1)$. Consequently, for the VRPSD, Algorithm~\ref{algorithm:network_separation} has a time complexity of~$\mathcal{O}(|\mathcal{A}|)$.

\begin{algorithm}[H]
\hspace*{\algorithmicindent} \textbf{Input:} $(\bar{\alpha}, \bar{\alpha}_0) \in \R^p \times \R_+$, network~$(\mathcal{N} = (\mathcal{V}, \mathcal{A}), b, u)$, spanning tree~$\mathcal{T}$ and basic feasible solution~$y^*$. \\
\hspace*{\algorithmicindent} \textbf{Output:} Set of rays~$R''$ such that~$\bar{\alpha}^\T (Qr) + \bar{\alpha}_0 (d^\T r) < 0$, for all~$r \in R''$.
\begin{algorithmic}[1]
\Procedure {FindViolatedRay}{$\bar{\alpha}, \bar{\alpha}_0, \mathcal{N} = (\mathcal{V}, \mathcal{A}), b, u, \mathcal{T}, y^*$}
    \State{$R'' \gets \emptyset$}
    \State \multiline{Let~$\oa{\mathcal{T}}$ be the arc set of an arborescence obtained by rooting~$\mathcal{T}$ at an arbitrary vertex~$s$.}
    \State {Let~$v_1 = s, v_2, \ldots, v_{|\mathcal{V}|}$ be an ordering obtained with a preorder traversal of~$\oa{\mathcal{T}}$.}
    \State {$\nu_s \gets 0$}
    \For {$j = 2, \ldots, |\mathcal{V}|$}
        \State {Let~$i < j$ be such that~$(v_i, v_j) \in \oa{\mathcal{T}}$.}
        \If {$(v_i, v_j) \in \mathcal{T}$}
            \State {$\nu_{v_j} \gets \nu_{v_i} + \bar{\alpha}^\T Q_a + \bar{\alpha}_0 \: d_a$}
        \Else
            \State {$\nu_{v_j} \gets \nu_{v_i} - \bar{\alpha}^\T Q_a - \bar{\alpha}_0 \: d_a$}
        \EndIf
    \EndFor
    \For {$uv \in \mathcal{A} \setminus \mathcal{T}$}
        \If { $\nu_{u} - \nu_v + \bar{\alpha}^\T Q_{uv} + \bar{\alpha}_0 \: d_{uv} < 0$ }
            \State{$R'' \gets R'' \cup \{r^{uv}\}$}
        \EndIf
    \EndFor
    \State {\textbf{return}~$R''$}
    \EndProcedure
\end{algorithmic}
\caption{Separation of Corner Benders' Cuts for Network Flow Subproblems}
\label{algorithm:network_separation}
\end{algorithm}

The intuition behind Algorithm~\ref{algorithm:network_separation} is as follows: since the set of rays~$R$ corresponds to the cycles formed by adding an arc from~$\mathcal{A} \setminus \mathcal{T}$ to the spanning tree~$\mathcal{T}$, we can precompute the \textit{potentials}~$\nu \in \R^{\mathcal{V}}$ along~$\mathcal{T}$. In other words, we set~$\nu_v$ to be the cost (with respect to weights~$d' = Q^\T \bar{\alpha} + \bar{\alpha}_0 \, d$) of the unique path in~$\mathcal{T}$ connecting~$s$ and~$v$. For a given arc~$uv \in \mathcal{A}$, it is straightforward to verify that~$\nu_u - \nu_v + \bar{\alpha}^\T Q_{uv} + \bar{\alpha}_0 \, d_{uv} = \alpha^\T (Q r^{uv}) + \alpha_0 (d^\T r^{uv})$. Thus, Algorithm~\ref{algorithm:network_separation} efficiently solves the separation problem for the ray inequalities in the polar set~\eqref{epigraph_rev_polar}.

\section{
    \mytitle{Solving the Lagrangian cut generating problem with a linear program containing only} \texorpdfstring{$y$}{y} \mytitle{variables}
}
\label{appendix:solve_lagrangian}

Suppose that~$X = \{x \in \R^n_+: (A_x) x \geq b_x\}$ and~$Y = \{y \in \R^m_+: (A_y) y \geq b_y\}$. In addition, assume that~$T = -I$,~$h = \mathbf{0}$ and~$Q \in \R_+^{n \times m}$ (all entries of~$Q$ are nonnegative). By Theorem~\ref{thm:lagrangian_cut}, to obtain a Lagrangian cut we need to solve
\begin{equation}
    \max_{\alpha \in \R^n} \left\{\sigma_{X}(c - \alpha) + \sigma_{Y}(Q^\T \alpha + d)\right\}.
    \label{problem:lagrangian3}
\end{equation}

\noindent
One way to solve Problem~\eqref{problem:lagrangian3} is to first rewrite Problem~\eqref{problem:basic} to get the following primal-dual pair.

\mbox{}\\
\begin{minipage}{.45\linewidth}
    \begin{equation}
    \begin{aligned}
    \min_{x,y}~~& c^\T x + d^\T y \\
    \textsuchthat{} & x = Qy, & (\alpha) \\
    & (A_x) x \geq b_x, & (\gamma) \\
    & (A_y) y \geq b_y, & (\nu) \\
    & x, y \geq 0.
    \end{aligned}
    \label{problem:full_LP1}
    \end{equation}
\end{minipage}%
\begin{minipage}{.45\linewidth}
    \begin{equation}
    \begin{aligned}
    \max_{\gamma,\nu}~~& \gamma^\T b_x + \nu^\T b_y \\
    \textsuchthat{} & \alpha +(A_x)^\T \gamma \leq c, \\
    & -Q^\T \alpha + (A_y)^\T \nu \leq d, \\
    & \gamma, \nu \geq 0.
    \end{aligned}
    \label{problem:full_LP1_dual}
    \end{equation}
\end{minipage}

\mbox{}\\
It follows from LP duality that if~$(\bar{\alpha}, \bar{\gamma}, \bar{\nu})$ is optimal for Problem~\eqref{problem:full_LP1_dual}, then~$\bar{\alpha}$ is optimal for Problem~\eqref{problem:lagrangian3} (see~\cite{frangioni2005lagrangian} for a nice proof via ``partial dualization'').

Now substitute~$x = Qy$ in Problem~\eqref{problem:full_LP1} to get

\mbox{}\\
\begin{minipage}{.5\linewidth}
    \begin{equation}
    \begin{aligned}
    \min_{y}~~& (Q^\T c + d)^\T y \\
    \textsuchthat{} & (A_x) Qy \geq b_x, & (\gamma) \\
    & (A_y) y \geq b_y, & (\nu) \\
    & y \geq 0.
    \end{aligned}
    \label{problem:full_LP2}
    \end{equation}
\end{minipage}%
\begin{minipage}{.5\linewidth}
    \begin{equation}
    \begin{aligned}
    \max_{\gamma,\nu}~~& \gamma^\T b_x + \nu^\T b_y \\
    \textsuchthat{} & Q^\T (A_x)^\T \gamma + (A_y)^\T \nu \leq Q^\T c + d, \\
    & \gamma, \nu \geq 0.
    \end{aligned}
    \label{problem:full_LP2_dual}
    \end{equation}
\end{minipage}

\mbox{}\\
Let~$(\hat{\gamma}, \hat{\nu})$ be optimal for Problem~\eqref{problem:full_LP2_dual}. Set~$\hat{\alpha} = c - (A_x)^\T \hat{\gamma}$ and notice that~$(\hat{\alpha}, \hat{\gamma}, \hat{\nu})$ is optimal for Problem~\eqref{problem:full_LP1_dual}, as desired.

\section{\mytitle{On our implementation of Parada et al. (2024)}}
\label{appendix:parada}

We now describe a minor modification that we made to the algorithm proposed by~\cite{parada2024disaggregated} to address a small issue in their original description. For that, we briefly review the inequalities they use and their separation routines for fractional solutions. 

The DL-shaped method proposed by~\cite{parada2024disaggregated} separates two types of ILS inequalities: \emph{P-cuts} and \emph{S-cuts}. Given an elementary route~$R = (v_1, \ldots, v_\ell)$, a P-cut has the form
\begin{equation}
    \label{ineq:P-cut}
    \sum_{i = 1}^\ell \theta'_{v_i} \geq \mb{E}[\mathcal{Q}(R)] \cdot (1 + (x_{v_1v_2} - 1) + \cdots + (x_{v_{\ell - 1}v_\ell} - 1)).
\end{equation}
Given~$\emptyset \subsetneq S \subseteq V_+$, an S-cut has the form
\begin{equation}
    \label{ineq:S-cut}
    \sum_{v \in S} \theta'_v \geq L(S) \cdot \left(1 + \left(\sum_{uv \in E: u, v \in S} x_{uv} \right) - |S| + \left\lceil \frac{\bar{q}(S)}{C} \right\rceil \right),
\end{equation}
where~$L(S)$ is a \emph{recourse lower bound} as defined in~\cite{parada2024disaggregated} (we use the lower bounds~$L_1$ and~$L_2$ from their work).

Let~$(\bar{x}, \bar{\theta}')$ be a fractional solution for~(\textsc{VRPSD-LP}) and consider Line~\ref{line:algorithm_rci_ils} of Algorithm~\ref{algorithm:vrpsd}.
The separation algorithm of~\cite{parada2024disaggregated} first attempts to find violated S-cuts associated with the customer sets generated by CVRPSEP during the separation of RCIs. It then searches for additional violated S-cuts and P-cuts by examining the graph~$G(\bar{x})$, which is associated with the edges in the support of~$\bar{x}$ that are not incident to the depot, that is,~$V(G(\bar{x})) = V_+$ and~$E(G(\bar{x})) = \{e \in E \setminus \delta(0) : \bar{x}_e > 0\}$. For each connected component~$\mathcal{H}$ of~$G(\bar{x})$, the algorithm verifies if the S-cut for~$S = V(\mathcal{H})$ is violated by~$(\bar{x}, \bar{\theta}')$, and if so, it adds the inequality to the model. \cite{parada2024disaggregated} then adds a P-cut whenever~$|V(\mathcal{H})| = |E(\mathcal{H})| + 1$, since they claim that in this case~$\mathcal{H}$ is a path.

However, this last claim is not always correct, so we instead explicitly verify whether~$\mathcal{H}$ is a path by traversing its nodes. As a counterexample, consider the fractional solution~$\bar{x}$ in Figure~\ref{figure:parada_counter_example}. Here,~$G(\bar{x})$ is made of a single connected component~$\mathcal{H}$ that contains a vertex with degree 3, so it is not a path despite satisfying~$|V(\mathcal{H})| = |E(\mathcal{H})| + 1$.

\begin{figure}[htb]
    \centering
    \includegraphics[scale=1]{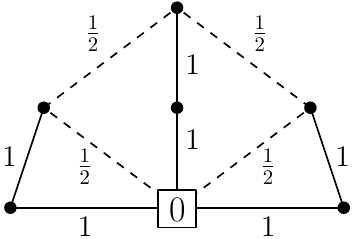}
    \caption{Example of a fractional solution~$\bar{x}$ where the associated connected component~$\mathcal{H}$ satisfies~$|V(\mathcal{H})| = |E(\mathcal{H})| + 1$ but is not a path.}
    \label{figure:parada_counter_example}
\end{figure}

\noindent
We remark, however, that the issue illustrated in Figure~\ref{figure:parada_counter_example} does not seem to occur often, and our preliminary experiments suggest that our fix does not lead to any significant performance improvement.

Figures~\ref{figure:gap_parada} and~\ref{figure:time_parada} compare our implementation of the algorithm from~\cite{parada2024disaggregated} with the original results reported by the authors. These figures are generated in the same way as Figure~\ref{figure:experiments}, except that here we consider all instances together. The curve labeled~\textsc{Parada}* corresponds to the results from the authors' online table, while~\textsc{Parada} denotes our implementation. We note that our implementation achieves comparable execution times but consistently obtains smaller root gaps. This improvement is likely due to our custom cutting-plane loop at the root node of the branch-and-bound tree, as described in Section~\ref{subsection:implementation} and detailed in Algorithm~\ref{algorithm:vrpsd}.

\begin{figure}[htb!]
    \centering
    \subfloat[Gap for all instances.]{
        \includegraphics[width=0.46\linewidth]{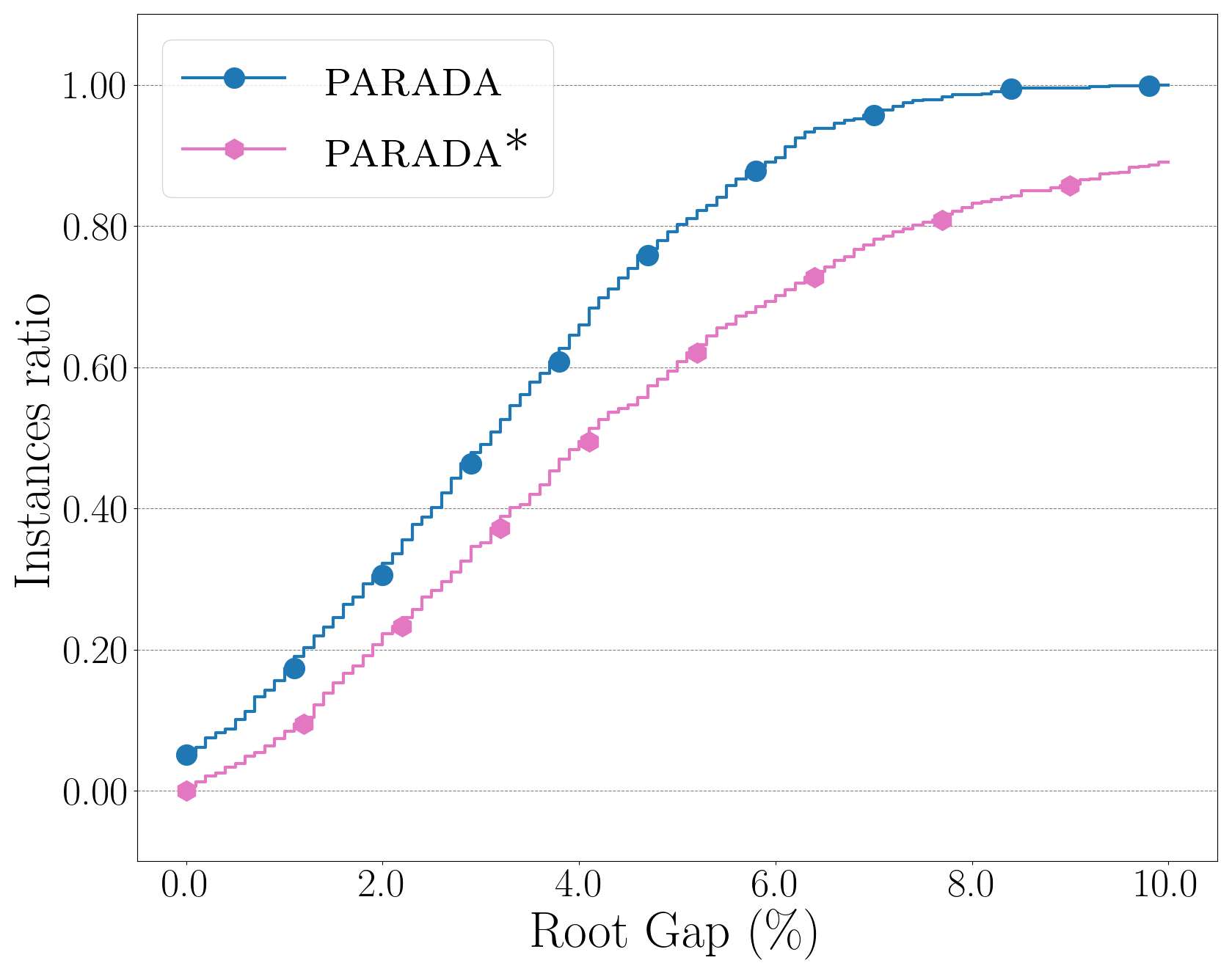}
        \label{figure:gap_parada}
    }
    \hfill
    \subfloat[Time for all instances.]{
        \includegraphics[width=0.46\linewidth]{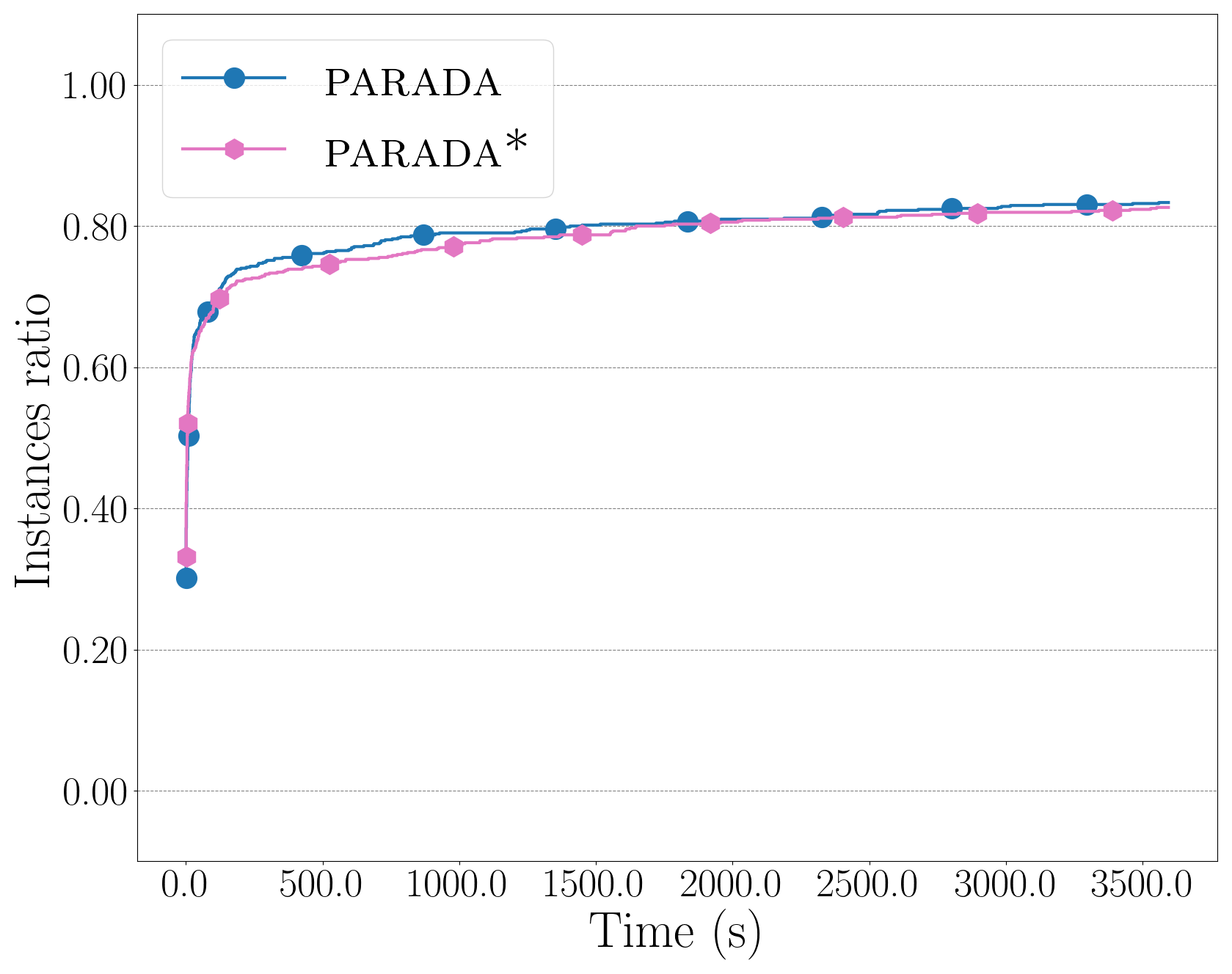}
        \label{figure:time_parada}
    }
    \\[0.5cm]
    \caption{
        Comparison between our implementation of the algorithm from~\cite{parada2024disaggregated} and the results reported in their online table.
    }
    \label{figure:experiments_parada}
\end{figure}

\section{\mytitle{Experiments on the separation of corner Benders' cuts}}
\label{appendix:experiments_separation_rays}

Here we examine the efficiency of Algorithm~\ref{algorithm:solve_polar} (i.e.~\solvereversepolar) to separate corner Benders' cuts. For the same instance as in Figure~\ref{figure:convergence}, we show in Figure~\ref{figure:separation} how much time we take to separate each Benders' cut. For example, a point~$p = (p_1, p_2) \in \Z_+ \times \R_+$ in the line of \tsc{corner} (time) indicates that \tsc{corner} took~$p_2$ seconds to separate the~$p_1$-th corner Benders' cut. Similarly, a point~$p = (p_1, p_2) \in \Z^2_+$ in the line of \tsc{corner} (rays) indicates that to separate the~$p_1$-th corner Benders' cut, \tsc{SolvePolar} separated~$p_2$ ray inequalities. To compare, we also show in Figure~\ref{figure:separation} the time that \tsc{benders} took to separate each Benders' cut.

\begin{figure}
    \centering
    \includegraphics[scale=0.22]{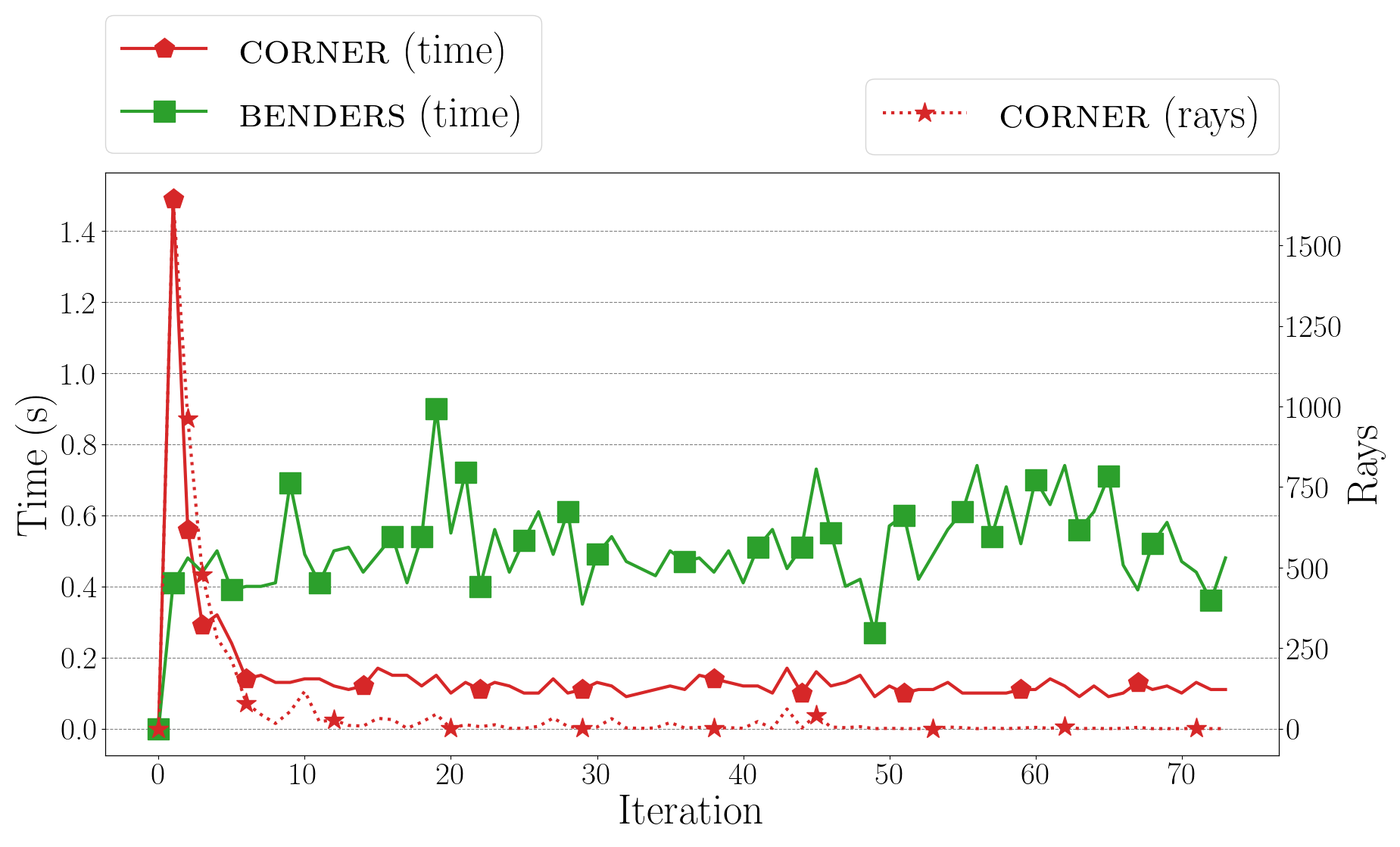}
    \caption{The solid lines indicate the time to separate each (corner) Benders' cuts. The dotted line shows how many ray inequalities were separated by~\solvereversepolar for each corner Benders' cut. Experiments were executed on instance ``30\_7\_0.95\_1.00\_8'' of~\cite{parada2024disaggregated} ($|V| = 30$ and~$k = 7$).}
    \label{figure:separation}
\end{figure}

Algorithm \tsc{corner} only took longer than \tsc{benders} when separating the first two corner Benders' cuts. Additionally, the dotted line illustrates that the number of ray inequalities separated by~\solvereversepolar rapidly decreases, highlighting the effectiveness of our ``warm-start'' strategy, which is encoded in the choice of the set $\tilde{R}$. 

It is also worth noting that the network $\mathcal{N} = (\mathcal{V}, \mathcal{A})$ for this instance contains 608 nodes and 12647 arcs. Consequently, the optimal corner $C = \{y^*\} + \cone(R)$ used in our approach has $|R| = |\mathcal{A}| - |\mathcal{V}| + 1 = 12040$ rays. Summing the number of rays separated in each iteration of \tsc{corner}, we find that a total of~$4405$ ray inequalities were separated. This collaborates with our observation in Section~\ref{subsection:corner_separation} that the projection $r \to (Qr, d^\T r)$ makes many ray inequalities redundant. By generating these ray inequalities as cutting planes, we can efficiently generate facet-defining inequalities for~$\epi(f_C)$ without having to repeatedly solve linear programs that have~$|R| + 1$ rows.

\end{APPENDICES}

\end{document}